\documentclass[11pt,oneside,english]{article}
\usepackage{amsmath, amsthm, amssymb, amsfonts, enumerate, enumitem, xspace, mathtools}
\usepackage[T1]{fontenc}
\usepackage[latin9]{inputenc}
\usepackage{geometry}
\usepackage{amstext}
\usepackage{amsthm}
\usepackage{amsmath,amsfonts,amsthm,amssymb,mathrsfs,dsfont} 
\usepackage{mathtools}
\usepackage{graphicx}
\usepackage{xcolor}
\usepackage{bbm}
\usepackage[colorinlistoftodos,textsize=scriptsize,backgroundcolor=orange!20]{todonotes}
\usepackage[colorlinks=true, allcolors=blue]{hyperref}
\usepackage[capitalize,nameinlink]{cleveref}
\usepackage{verbatim}
\allowdisplaybreaks

\usepackage{setspace}

\makeatletter
\newtheorem{theorem}{Theorem}[section]

\newcommand{\theoremname}{testing}

\newtheorem{lemma}[theorem]{Lemma}

\newtheorem{claim}[theorem]{Claim}

\newtheorem{proposition}[theorem]{Proposition}

\newtheorem*{question*}{Question}

\theoremstyle{definition}
\newtheorem{definition}[theorem]{Definition}

\newtheorem{remark}[theorem]{Remark}

\theoremstyle{plain}

\DeclareFontFamily{U}{mathc}{}
\DeclareFontShape{U}{mathc}{m}{it}%
{<->s*[1.03] mathc10}{}
\DeclareMathAlphabet{\mathscr}{U}{mathc}{m}{it}

\def \EE {\mathcal{E}}

\def \cZ {\mathcal{Z}}

\newcommand{\Rka}{R_k^{\alpha}}

\DeclareMathOperator{\supp}{supp}

\makeatother

\usepackage{babel}

\providecommand{\theoremname}{Theorem}

\DeclareMathOperator{\Span}{span}
\DeclareMathOperator{\Bad}{\boldsymbol{B}}
\DeclareMathOperator{\M}{\mathcal{M}}

\title{On the counting problem in inverse Littlewood--Offord theory}
\author{Asaf Ferber 
\thanks{Massachusetts Institute of Technology. Department of Mathematics. Email: {\tt ferbera@mit.edu}.}
\and Vishesh Jain \thanks{Massachusetts Institute of Technology. Department of Mathematics. Email: {\tt visheshj@mit.edu}. 
} \and Kyle Luh \thanks{Harvard University, Center of Mathematical Sciences and Applications.  Email: {\tt kluh@cmsa.fas.harvard.edu}.  
} \and Wojciech Samotij\thanks{School of Mathematical Sciences, Tel Aviv University, Tel Aviv 6997801, Israel. Email: {\tt samotij@tauex.tau.ac.il}. 
}}
\date{}

\begin{document}
\maketitle
\global\long\def\R{\mathbb{R}}

\global\long\def\S{\mathcal{R}}

\global\long\def\Z{\mathbb{Z}}

\global\long\def\C{\mathbb{C}}

\global\long\def\Q{\mathbb{Q}}

\global\long\def\N{\mathbb{N}}

\global\long\def\P{\Pr}

\global\long\def\F{\mathbb{F}}

\global\long\def\U{\mathcal{U}}

\global\long\def\V{\mathcal{V}}

\global\long\def\E{\mathbb{E}}

\global\long\def\rk{\text{rank}}

\global\long\def\A{\mathcal{A}}

\global\long\def\L{\mathcal{L}}

\global\long\def\QQ{\mathcal{Q}}

\def \a {\alpha}
\def \b {\beta}
\def \g {\gamma}
\def \e {\varepsilon}
\def \d {\delta}
\def \l {\lambda}

\begin{abstract}
  Let $\epsilon_1, \dotsc, \epsilon_n$ be i.i.d.\ Rademacher random variables taking values $\pm 1$ with probability $1/2$ each. Given an integer vector $\boldsymbol{a} = (a_1, \dotsc, a_n)$, its concentration probability is the quantity $\rho(\boldsymbol{a}):=\sup_{x\in \Z}\Pr(\epsilon_1 a_1+\dotsb+\epsilon_n a_n = x)$. The Littlewood--Offord problem asks for bounds on $\rho(\boldsymbol{a})$ under various hypotheses on $\boldsymbol{a}$, whereas the \emph{inverse} Littlewood--Offord problem, posed by Tao and Vu, asks for a characterization of all vectors $\boldsymbol{a}$ for which $\rho(\boldsymbol{a})$ is large. In this paper, we study the associated counting problem: \emph{How many} integer vectors $\boldsymbol{a}$ belonging to a specified set have large $\rho(\boldsymbol{a})$? The motivation for our study is that in typical applications, the inverse Littlewood--Offord theorems are only used to obtain such counting estimates. Using a more direct approach, we obtain significantly better bounds for this problem than those obtained using the inverse Littlewood--Offord theorems of Tao and Vu and of Nguyen and Vu. Moreover, we develop a framework for deriving upper bounds on the probability of singularity of random discrete matrices that utilizes our counting result. To illustrate the methods, we present the first `exponential-type' (i.e., $\exp(-n^c)$ for some positive constant $c$) upper bounds on the singularity probability for the following two models: (i) adjacency matrices of dense signed random regular digraphs, for which the previous best known bound is $O(n^{-1/4})$ due to Cook; and (ii) dense row-regular $\{0,1\}$-matrices, for which the previous best known bound is $O_{C}(n^{-C})$ for any constant $C>0$ due to Nguyen.
\end{abstract}

\section{Introduction}
\subsection{Littlewood--Offord theory}
In connection with their study of random polynomials, Littlewood and Offord~\cite{littlewood1943number} introduced the following problem. Let $\boldsymbol{a}:= (a_1,\dotsc, a_n) \in (\Z \setminus \{0\})^{n}$ and let $\epsilon_1,\dotsc, \epsilon_n$ be independent and identically distributed (i.i.d.) Rademacher random variables, i.e., each $\epsilon_i$ independently takes values $\pm 1$ with probability $1/2$ each. Estimate the largest atom probability $\rho(\boldsymbol{a})$, which is defined by
\[
  \rho(\boldsymbol{a}) := {\textstyle \sup_{x\in \Z}}\Pr\left(\epsilon_1 a_1 + \dotsb + \epsilon_n a_n = x\right).
\]
They showed that $\rho(\boldsymbol{a}) = O(n^{-1/2}\log{n})$ for any such $\boldsymbol{a}$. Soon after, Erd\H{o}s~\cite{erdos1945lemma} used Sperner's theorem to give a simple combinatorial proof of the refinement $\rho(\boldsymbol{a}) \leq \binom{n}{\lfloor n/2 \rfloor} / 2^{n} = O(n^{-1/2})$, which is tight, as is readily seen by taking $\boldsymbol{a}$ to be the all ones vector. 

The results of Littlewood--Offord and Erd\H{o}s generated considerable interest and inspired further research on this problem. One such direction of research was concerned with improving the bound of Erd\H{o}s under additional assumptions on $\boldsymbol{a}$. The first such improvement was due to Erd\H{o}s and Moser~\cite{erdos1947e736}, who showed that if all coordinates of $\boldsymbol{a}$ are distinct, then $\rho(\boldsymbol{a}) = O(n^{-3/2}\log{n})$. Subsequently, S\'ark\H{o}zy and Szemer\'edi~\cite{sarkozy1965problem} improved this estimate to $O(n^{-3/2})$, which is asymptotically optimal. Soon afterwards, Hal\'asz~\cite{halasz1977estimates} proved the following very general theorem relating the `additive structure' of the coordinates of $\boldsymbol{a}$ to $\rho(\boldsymbol{a})$. 

\begin{theorem}[Hal\'asz~\cite{halasz1977estimates}] 
  \label{thm:halasz-orig}
  Let $\boldsymbol{a}:=(a_1,\dots,a_n) \in (\Z \setminus \{0\})^{n}$. For an integer $k \geq 1$, let $R_k(\boldsymbol{a})$ denote the number of solutions to $\pm a_{i_1} \pm a_{i_2} \dotsb \pm a_{i_{2k}} = 0$, where repetitions are allowed in the choice of $i_1,\dots,i_{2k} \in [n]$. There exists an absolute constant $C > 0$ such that
  \[
    \rho(\boldsymbol{a})\leq \frac{C\sqrt{k}R_{k}(\boldsymbol{a})}{2^{2k} n^{2k+1/2}} + e^{-n/\max\{k,C\}}.
  \]
\end{theorem}

It is easy to see that Hal\'asz's inequality, applied with $k=1$, yields the estimate $\rho(\boldsymbol{a}) = O(n^{-1/2})$ for every $\boldsymbol{a} \in (\Z \setminus \{0\})^n$; if one further assumes that the coordinates of $\boldsymbol{a}$ are distinct, then $R_1(\boldsymbol{a}) \le 2n$ and one obtains the stronger bound $\rho(\boldsymbol{a}) = O(n^{-3/2})$, recovering the result of S\'ark\H{o}zy and Szemer\'edi. We emphasize that \cref{thm:halasz-orig} is valid even when $k$ grows with $n$ (the constant $C$ does not depend on either $k$, $n$, or $\boldsymbol{a}$). This fact will prove to be crucial for our work.  

\subsection{Inverse Littlewood--Offord theory}

Guided by inverse theorems from additive combinatorics, Tao and Vu \cite{tao2009inverse} brought a new perspective to the Littlewood--Offord problem. Instead of imposing further assumptions on $\boldsymbol{a}$ in order to obtain better bounds on $\rho(\boldsymbol{a})$, they tried to find the underlying reason why $\rho(\boldsymbol{a})$ could be large. In this subsection, we provide only a very brief overview of their findings and of subsequent work that followed. We refer the interested reader to the survey \cite{nguyen2013small} and the textbook \cite{tao2006additive} for further information on both forward and inverse Littlewood--Offord theory. We begin by recalling a central notion in additive combinatorics. 

\begin{definition}
  \label{defn:GAP}
  For an integer $r \geq 0$, we say that a set $Q \subseteq \Z$ is a \emph{generalized arithmetic progression (GAP)} of \emph{rank} $r$ if
  \[
    Q:= \{q_0 + x_1 q_1 + \dotsb + x_r q_r : x_i \in \Z, M_i \leq x_i \leq M_i'\text{ for all }i \in [r]\},
  \]
for some $q_0,\dotsc, q_r, M_1,\dotsc, M_r, M_1',\dotsc, M_r' \in \Z$. The numbers $q_i$ are called the \emph{generators} of $Q$. If $M_i = -M_i'$ for all $i \in [r]$ and $q_0 = 0$, then $Q = -Q$ and thus $Q$ is said to be \emph{symmetric}.  
\end{definition}

It is often useful to think of $Q$ as the image of the integer box $B:= \{(x_1,\dotsc,x_r) \in \Z^{r} : M_i \leq x_i \leq M_i'\}$ under the affine map 
\[
  \Phi \colon (x_1,\dots,x_r) \mapsto q_0 + x_1 q_1 + \dotsb + x_r q_r.
\]
If $\Phi$ is an injective map, we say that $Q$ is \emph{proper}. In this case, we also define the \emph{volume} of $Q$ to be the cardinality of $B$ (which is equal to the cardinality of $Q$). \\

Returning to the Littlewood--Offord problem, it is easy to see that if the coordinates  of $\boldsymbol{a}$ belong to a proper symmetric GAP of `small' rank and `small' volume, then $\rho(\boldsymbol{a})$ is necessarily `large'. More precisely, fix an $r$ and suppose that there are integers $q_1, \dotsc, q_r$ and $M_1, \dotsc, M_r$ such that $a_i = x_{i,1} q_1 + \dotsb + x_{i,r} q_r$, where $|x_{i,j}| \le M_j$, for all $i \in [n]$ and $j \in [r]$. In this case, the random sum $S:= \epsilon_1 a_1 + \dotsb + \epsilon_n a_n$ may be written as
\[
  S = q_1 \cdot \left\{\epsilon_1 x_{1,1} +\dotsb + \epsilon_n x_{n,1}\right\} + \dotsb + q_r \cdot \left\{\epsilon_1 x_{1,r} + \dotsb + \epsilon_n x_{n,r}\right\}.
\]
It follows from Chebyshev's inequality that with probability at least $1/2$, each of the $r$ sums $\epsilon_1 x_{1,j} + \dotsb + \epsilon_n x_{n,j}$ falls into an interval of length $O_r(\sqrt{n} M_j)$. Letting $B = \{-M_1, \dotsc, M_1\} \times \dotsb \times \{-M_r, \dotsc, M_r\}$, we may conclude that with probability at least $1/2$, the variable $S$ takes values in a fixed subset of size at most $O_r( n^{r/2}|B|)$. By the pigeonhole principle, there is some value which $S$ assumes with probability at least $\Omega_r(n^{-r/2}|B|^{-1})$. In other words, we see that
\[
  \rho(\boldsymbol{a}) = \Omega_{r}\left(\frac{1}{n^{r/2}|B|}\right).
\]

In particular, if the coordinates of an $n$-dimensional vector $\boldsymbol{a}$ are contained in a  GAP of rank $r$ and volume at most $n^{C-r/2}$, for some constant $C$, then $\rho(\boldsymbol{a}) = \Omega_{r}(n^{-C})$. The inverse Littlewood--Offord theorems of Tao and Vu~\cite{tao2009inverse,tao2010sharp} use deep Freiman-type results from additive combinatorics to show that a weak converse of this statement holds. Roughly speaking, the only reason for a vector $\boldsymbol{a}$ to have $\rho(\boldsymbol{a})$ only polynomially small is that most coordinates of $\boldsymbol{a}$ belong to a GAP of small rank and small volume. These results were subsequently sharpened by Nguyen and Vu~\cite{nguyen2011optimal}, who proved the following optimal inverse Littlewood--Offord theorem. 

\begin{theorem}[Nguyen--Vu~\cite{nguyen2011optimal}]
\label{thm:optimal-LO-nv}
Let $C$ and $\varepsilon < 1$ be positive constants. If $\boldsymbol{a} \in (\Z \setminus \{0\})^n$ satisfies
\[
  \rho(\boldsymbol{a}) \geq n^{-C},
\]
then there exists a proper symmetric GAP $Q$ of rank $r=O_{C,\varepsilon}(1)$ and volume 
\[
  |Q| = O_{C,\varepsilon}\left(\frac{1}{\rho(\boldsymbol{a})n^{r/2}}\right)
\]
that contains all but at most $\varepsilon n$ coordinates of $\boldsymbol{a}$ (counting multiplicities). 
\end{theorem}

We remark that Nguyen and Vu also proved a version of the above theorem (this is \cite[Theorem~2.5]{nguyen2011optimal}) whose statement allows for a trade-off between the size of the `exceptional set' of coordinates of $\boldsymbol{a}$ which are not in the GAP $Q$, and the bound on the size of $Q$.

\subsection{The counting problem in inverse Littlewood--Offord theory}

For typical applications, especially those in random matrix theory, one needs to resolve only the following \emph{counting variant} of the inverse Littlewood--Offord problem: for \emph{how many} vectors $\boldsymbol{a}$ in a given collection $\mathcal{A}\subseteq \Z^{n}$ is their largest atom probability $\rho(\boldsymbol{a})$ greater than some prescribed value? The utility of such results is that they enable various union bound arguments, as one can control the number of terms in the relevant union/sum. Such counting results may be easily deduced from the inverse Littlewood--Offord theorems, as we shall now show.

As a motivating example (see \cite{nguyen2011optimal}), suppose that we would like to count the number of integer vectors $\boldsymbol{a} \in \Z^{n}$ such that $\|\boldsymbol{a}\|_{\infty} \leq N = n^{O(1)}$ and $\rho(\boldsymbol{a}) \geq \rho:=n^{-C}$. \cref{thm:optimal-LO-nv} states that for any $\varepsilon \in (0,1)$, all but $\varepsilon n$ of the coordinates (counting multiplicities) of any such vector $\boldsymbol{a}$ are contained in a proper symmetric GAP $Q$ of rank $r = O_{C,\varepsilon}(1) \geq 1$ and volume $|Q|= O_{C,\varepsilon}(n^{C - \frac{r}{2}})$. Fix any such $Q$. The number of $n$-dimensional vectors all of whose coordinates belong to $Q$ is at most
\[
  |Q|^{n} \leq (O_{C,\varepsilon}(1))^{n} n^{Cn}n^{-\frac{n}{2}}.
\]
Moreover, there are at most $\binom{n}{\varepsilon n} \cdot N^{\varepsilon n} = n^{O(\varepsilon)n}$ ways to introduce the `exceptional' $\varepsilon n$ coordinates from outside of $Q$. Finally, a more detailed version of \cref{thm:optimal-LO-nv} states that the number of ways in which we can choose the proper symmetric GAP $Q$ is negligible compared to our bound on $|Q|^{n}$. To summarize, we see that the number of vectors $\boldsymbol{a}$ satisfying the properties at the start of this paragraph is at most
\[
  n^{n\left(C-\frac{1}{2}+O(\varepsilon) + o_{C,\varepsilon}(1)\right)}.
\]
It is not difficult to see that this is tight up to the $O(\varepsilon)+ o_{C,\varepsilon}(1)$ term in the exponent. 

The primary drawback of the structural approach to the counting problem, which we described above, is that it is only effective for counting vectors $\boldsymbol{a}$ with $\rho(\boldsymbol{a}) \geq n^{-C}$, where $C > 0$ is allowed to grow only very mildly (in particular, much slower than logarithmically) with $n$. This is due to the dependencies between $C$ and $\varepsilon$ and the constants implicit in the $O$-notation. To make matters worse, improving these dependencies would most likely require (among other things) improving the bounds in Freiman's theorem, which is one of the central unsolved problems in additive combinatorics. In contrast, for many applications, one would ideally like to count vectors $\boldsymbol{a}$ with even exponentially small values of $\rho(\boldsymbol{a})$. Our first main theorem is a counting result for the inverse Littlewood--Offord problem, which is effective for values of $\rho(\boldsymbol{a})$ as small as $\exp\left(-c\sqrt{n\log{n}}\right)$, where $c>0$ is some sufficiently small constant. In order to motivate and state it, we need some preparation.

The starting point for our approach is the anti-concentration inequality of Hal\'asz mentioned earlier (\cref{thm:halasz-orig}). For reasons which will become clear later, we shall work with a variant of this inequality for finite fields of prime order. For a vector $\boldsymbol{a} \in \F_p^n$, we define $\rho_{\F_p}(\boldsymbol{a})$ and $R_k(\boldsymbol{a})$ as in~\cref{thm:halasz-orig}, except that all arithmetic is done over the $p$-element field $\F_p$, and we let $\supp(\boldsymbol{a}) = \{i \in [n] : a_i \neq 0 \mod p\}$.

\begin{theorem}[Hal\'asz's inequality over $\F_p$] 
  \label{thm:halasz-fp}
  There exists an absolute constant $C$ such that the following holds for every odd prime $p$, integer $n$, and vector $\boldsymbol{a}:=(a_1,\dotsc, a_n) \in \F_p^{n}\setminus \{\boldsymbol{0}\}$. Suppose that an integer $k \ge 0$ and positive real $M$ satisfy $30M \leq |\supp(\boldsymbol{a})|$ and $80kM \leq n$. Then,
  \[
    \rho_{\F_p}(\boldsymbol{a})\leq \frac{1}{p}+\frac{CR_k(\boldsymbol{a})}{2^{2k} n^{2k} \cdot M^{1/2}} + e^{-M}.
  \]
\end{theorem}

The proof of this theorem is a straightforward adaptation of Hal\'asz's original argument from~\cite{halasz1977estimates}. For the reader's convenience, we provide complete details in \cref{app:halasz}.

Note that Hal\'asz's inequality may be viewed as a \emph{partial inverse Littlewood--Offord theorem}. Indeed, if $\rho_{\F_p}(\boldsymbol{a})$ is `large', then it must be the case that $R_k(\boldsymbol{a})$ is also `large'. Hence, an upper bound on the number of vectors $\boldsymbol{a}$ for which $R_k(\boldsymbol{a})$ is `large' is also an upper bound on the number of vectors with `large' $\rho_{\F_p}(\boldsymbol{a})$. Moreover, since $\rho_{\F_p}(\boldsymbol{a}) \le \rho_{\F_p}(\boldsymbol{b})$ for every subvector $\boldsymbol{b} \subseteq \boldsymbol{a}$, when $\rho_{\F_p}(\boldsymbol{a})$ is `large', so is $R_k(\boldsymbol{b})$ for \emph{every} $\boldsymbol{b} \subseteq \boldsymbol{a}$. As we shall show, the number of vectors $\boldsymbol{a}$ with such `hereditary' property can be bounded from above quite efficiently using direct combinatorial arguments. Consequently, our approach yields strong bounds on the number of vectors $\boldsymbol{a}$ with $\rho_{\F_p}(\boldsymbol{a}) \ge \rho$ for a significantly wider range of $\rho$ than the range amenable to the `structural' approach described above.

Instead of working directly with $R_k(\boldsymbol{a})$, however, we will find it more convenient to work with the following closely related quantity.

\begin{definition}
  Suppose that $\boldsymbol{a}\in \F_{p}^{n}$ for an integer $n$ and a prime $p$ and let $k \in \N$. For every $\alpha \in [0,1]$, we define $\Rka(\boldsymbol{a})$ to be the number of solutions to
  \[
    \pm a_{i_1}\pm a_{i_2}\dotsb \pm a_{i_{2k}}= 0 \mod p
  \]
  that satisfy $|\{i_1, \dotsc, i_{2k}\}| \ge (1+\alpha)k$.
\end{definition}

It is easily seen that $R_k(\boldsymbol{a})$ cannot be much larger than $\Rka(\boldsymbol{a})$. This is formalized in the following simple lemma.

\begin{lemma}
\label{lemma:R_k vs Rka}
For all integers $k$, $n$ with $k \le n/2$, any prime $p$, vector $\boldsymbol{a} \in \F^n_p$, and $\alpha \in [0,1]$,
\[
  R_k(\boldsymbol{a})\leq  \Rka(\boldsymbol{a}) + \left(40 k^{1-\alpha}n^{1+\alpha}\right)^k.
\]
\end{lemma}
\begin{proof}
  By definition, $R_k(\boldsymbol{a})$ is equal to $\Rka(\boldsymbol{a})$ plus the number of solutions to $\pm a_{i_1}\pm a_{i_2}\dotsb\pm a_{i_{2k}} = 0$ that satisfy $|\{i_1, \dotsc, i_{2k}\}| < (1+\alpha)k$. The latter quantity is bounded from above by the number of sequences $(i_1, \dotsc, i_{2k}) \in [n]^{2k}$ with at most $(1+\alpha)k$ distinct entries times $2^{2k}$, the number of choices for the $\pm$ signs. Thus
  \[
    R_k(\boldsymbol{a}) \leq \Rka(\boldsymbol{a}) + \binom{n}{(1+\alpha)k} \big((1+\alpha)k\big)^{2k}2^{2k} \leq \Rka(\boldsymbol{a}) +  \left(4e^{1+\alpha}k^{1-\alpha}n^{1+\alpha}\right)^k,
  \]
  where the final inequality follows from the well-known bound $\binom{a}{b} \le (ea/b)^b$. Finally, noting that $4e^{1+\alpha} \leq 4e^{2} \leq 40$ completes the proof.
\end{proof}

Our counting theorem provides an upper bound on the number of sequences $\boldsymbol{a}$ for which every `relatively large' subsequence $\boldsymbol{b}$ has `large' $\Rka(\boldsymbol{b})$. In particular, the sequences $\boldsymbol{a}$ that are not counted have a `relatively large' subsequence $\boldsymbol{b}$ with `small' $\Rka(\boldsymbol{b})$ and thus also `small' $R_k(\boldsymbol{b})$ (by \cref{lemma:R_k vs Rka}), and hence small $\rho_{\F_p}(\boldsymbol{b})$ (by \cref{thm:halasz-fp}). Since $\rho_{\F_p}(\boldsymbol{a}) \le \rho_{\F_p}(\boldsymbol{b})$ whenever $\boldsymbol{b} \subseteq \boldsymbol{a}$, each sequence $\boldsymbol{a}$ that is not counted has `small' $\rho_{\F_p}(\boldsymbol{a})$.

\begin{theorem}
  \label{thm:counting-lemma}
  Let $p$ be a prime, let $k, n \in \N$, $s\in [n]$, $t\in [p]$, and let $\alpha \in (0,1)$. Denoting
  \[
    \Bad_{k,s,\geq t}^{\alpha}(n):= \left\{\boldsymbol{a} \in \F_{p}^{n} : R^{\alpha}_k(\boldsymbol{b})\geq t\cdot \frac{2^{2k} \cdot |\boldsymbol{b}|^{2k}}{p} \text{ for every } \boldsymbol{b}\subseteq \boldsymbol{a} \text{ with } |\boldsymbol{b}|\geq s\right\},
  \]
  we have
  \[
    |\Bad_{k,s,\geq t}^{\alpha}(n)| \leq \left(\frac{s}{n}\right)^{2k-1} (\alpha t)^{s-n} p^n.
  \]
\end{theorem}
\begin{remark}
We emphasize that both the statement as well as the proof of our counting theorem are facilitated by working over the finite field $\F_p$. The counting corollaries of the inverse Littlewood--Offord theorems (over the integers) require additional hypotheses (as in the sample application mentioned above) in order to limit the number of GAPs that one needs to consider.  
\end{remark}

\begin{remark}
  It is well known (see, e.g., \cite{nguyen2011optimal}) that the inverse Littlewood--Offord theorems are powerful enough to recover Hal\'asz's inequality (\cref{thm:halasz-orig}) only for \emph{fixed} (or very mildly growing) values of $k$. In contrast, our approach utilizes Hal\'asz's inequality to provide non-trivial counting results even for $k$ growing as fast as $\sqrt{n\log{n}}$.   
\end{remark}

\subsection{Applications to random matrix theory}

The singularity problem for random Rademacher matrices asks the following deceptively simple question. 
Let $A_n$ denote a random $n\times n$ matrix whose entries are independent and identically distributed (i.i.d.) Rademacher random variables, which take values $\pm 1$ with probability $1/2$ each. What is the probability $c_n$ that $A_n$ is singular?\footnote{The singularity question for random Rademacher matrices is essentially equivalent to the singularity question for random Bernoulli (uniform on \{0,1\}) matrices. More precisely, let $M_n$ denote the $n\times n$ random Rademacher matrix and let $M_n'$ denote the $n \times n$ random Bernoulli matrix. The following coupling shows that $|\det(M_n)|$ has the same distribution as $2^{n-1} |\det(M_{n-1}')|$. Starting with $M_n$, we can multiply a subset of columns and a subset of rows by $-1$ so as to turn the first row and the first column of the matrix into the all ones vector; this does not affect the absolute value of the determinant. Next, by subtracting the first row from each of the other rows, we can further ensure that the first column equals $(1,0,\dots,0)^{T}$; this does not change the absolute value of the determinant either. The determinant of the resulting matrix is precisely equal to the determinant of the bottom-right $n-1 \times n-1$ submatrix. Since the choice of signs with which to multiply the rows and columns of $M_n$ depends only on the entries in the first row and the first column, it is readily checked that each entry of the bottom $n-1 \times n-1$ submatrix is $0$ or $-2$ with equal probability, independent of all other entries.} Considering the event that two rows or two columns of $A_n$ are equal (up to a sign) gives
\[
  c_n \geq (1+o(1))n^{2}2^{1-n}.
\]
 It is widely conjectured that this bound is tight.  On the other hand, perhaps surprisingly, it is non-trivial even to show that $c_n$ tends to $0$ as $n$ goes to infinity. This was accomplished in the classical work of Koml\'os~\cite{komlos1967determinant} in 1967; he showed that $c_n = O\big(n^{-1/2}\big)$ using the Erd\H{o}s--Littlewood--Offord anti-concentration inequality. Subsequently, a breakthrough result due to  Kahn, Koml\'os, and Szemer\'edi in 1995 \cite{kahn1995probability} showed that $$c_n = O(0.999^{n}).$$ 
In a very recent and impressive work, Tikhomirov \cite{tikhomirov2018}, improving on intermediate results by Tao and Vu \cite{tao2007singularity} and  Bourgain, Vu, and Wood
\cite{bourgain2010singularity}, showed that
$$c_n \leq (2+o(1))^{-n},$$
thereby settling the above conjecture up to lower order terms. 

The singularity problem becomes significantly more difficult when one considers models of random matrices with dependencies between entries. In this work, we develop a framework utilizing \cref{thm:counting-lemma} to study the singularity probability of two models of discrete random $n\times n$ matrices which come from the theory of random graphs: the adjacency matrix of a random regular digraph (r.r.d.) with independent $\pm$ signs and the adjacency matrix of random left-regular bipartite graph, that is, a uniformly random balanced bipartite graph whose all `left' vertices have the same degree. The best known upper bound on the singularity probability in the first model is not even $n^{-1}$; it is achieved by combining Koml\'os's argument with additional combinatorial ideas. The best known upper bound on the singularity probability in the second model is $n^{-C}$, for any constant $C>0$; it is obtained using a nonstandard application of the optimal inverse Littlewood--Offord theorem. In each of these two cases, it is conjectured (\cite{cook2017singularity,nguyen2013singularity}) that the singularity probability is, in fact, exponentially small. While not entirely settling these conjectures, we will provide the first `exponential-type' (i.e.\ $\exp(-cn^{c})$ for some positive constant $c$) upper bounds on the singularity probability for these models. Moreover, the arguments we use for studying both these models are very similar, whereas previously, they were handled using quite different techniques. We discuss this in more detail below.

\subsubsection{Singularity of signed r.r.d.\ matrices}

Let $\M^{\pm}_{n,d}$ denote the set of all $n\times n$ matrices $M_n^{\pm}$ with entries in $\{-1,0,1\}$ which satisfy the constraints
\[
  d = \sum_{i=1}^{n}|M^{\pm}_{n,d}(i,k)| = \sum_{j=1}^{n}|M^{\pm}_{n,d}(k,j)|
\]
for all $k\in [n]$. The probability of singularity of a uniformly random element of $\M^{\pm}_{n,d}$ was studied by Cook~\cite{cook2017singularity} as a first step towards the investigation of the singularity probability of the adjacency matrix of a random regular digraph. In particular, he showed the following.

\begin{theorem}[Cook~\cite{cook2017singularity}]
  Assume that $C\log^{2}n \leq d \leq n$ for a sufficiently large constant $C>0$ and let $M^{\pm}_{n,d}$ be a uniformly random element of $\M^{\pm}_{n,d}$. Then,
  \[
    \Pr\left(M^{\pm}_{n,d} \text{ is singular}\right) = O\left(d^{-1/4}\right).
  \]
\end{theorem}

To the best of our knowledge, Cook's result is the first to show that such matrices are invertible asymptotically almost surely, that is, with probability tending to one as $n$, the size of the matrix, tends to infinity. However, the upper bound on the probability of singularity is very weak. Indeed, Cook conjectured that when $d = \lceil r n \rceil $ for some fixed $ 0 < r \leq 1$, then the probability that $M^{\pm}_{n,d}$ is singular should be exponentially small. We make progress towards this conjecture by providing the first `exponential-type' upper bound on the singularity probability.

\begin{theorem}
  \label{thm:singularity-rrd}
  Fix an $r \in (0,1]$. For every integer $n$, let $d = \lceil rn \rceil$ and let $M^{\pm}_{n,d}$ be a uniformly random element of $\M^{\pm}_{n,d}$. There exists a constant $c > 0$ such that for all sufficiently large $n$,
  \[
    \Pr\left(M^{\pm}_{n,d} \text{ is singular}\right) \leq \exp\left(-n^c\right).
  \]
\end{theorem}

\begin{remark}
  Our proof method could provide a similar conclusion for much smaller values of $d$ (in particular, for $d = \Omega(n^{1-\ell})$ for some absolute constant $0 < \ell < 1$). However, in order to minimize technicalities and emphasize the main ideas, we will only present details for the case $d = \Theta(n)$.    
\end{remark}

\begin{remark}
  If we were to replace the application of \cref{thm:counting-lemma} in our proof of \cref{thm:singularity-rrd} with the counting corollary of the recent optimal inverse Littlewood--Offord theorem over finite fields due to Nguyen and Wood~\cite[Theorem 7.3]{nguyen2018random}, we would be able to deduce only the much weaker bound
  \[
    \Pr\left(M^{\pm}_{n,d} \text{ is singular}\right) = O_{C}(n^{-C})
  \]
  for every positive constant $C$. It is interesting to note that, proving an upper bound of the form $O_{C}(n^{-C})$ on the singularity probability for this model as well as the next one essentially requires the \emph{optimal} inverse Littlewood--Offord theorem.  
\end{remark}

\subsubsection{Singularity of random row-regular matrices}

For an even integer $n$, let $\QQ_n$ denote the set of $n \times n$ matrices $Q_n$ with entries in $\{0,1\}$ that satisfy the constraint
\[
  \sum_{j=1}^{n}Q_n(i,j) = \frac{n}{2}
\]
for each $i\in [n]$. Notice that $Q_n$ may be viewed as the bipartite adjacency matrix of a bipartite graph with parts of size $n$ such that each vertex on the left has exactly $n/2$ neighbors on the right. The probability of singularity of a uniformly random element of $\QQ_n$ was studied by Nguyen~\cite{nguyen2013singularity} as a relaxation of the singularity problem for the adjacency matrix of a random regular (di)graph; we refer the reader to the discussion there for further details about the motivation for studying this model and the associated technical challenges. Nguyen showed that the probability that $Q_n$ is singular decays faster than any polynomial.

\begin{theorem}[Nguyen~\cite{nguyen2013singularity}]
  For every even integer $n$, let $Q_n$ be a uniformly random element of $\QQ_n$. For every constant $C$,
  \[
    \Pr(Q_n \text{ is singular}) = O_{C}\big(n^{-C}\big).
  \]
\end{theorem}

Nguyen further conjectured~\cite[Conjecture 1.4]{nguyen2013singularity} that the probability that $Q_n$ is singular is $(2 + o(1))^{-n}$; note that this is the probability that two rows of $Q_n$ are the same. We make progress towards this conjecture by providing an `exponential-type' upper bound on the probability of singularity. 

\begin{theorem}
  \label{thm:singularity-row-regular}
  For every even integer $n$, let $Q_n$ be a uniformly random element of $\QQ_n$. There exists a contant $c > 0$ such that for all sufficiently large $n$,
  \[
    \Pr(Q_n \text{ is singular}) \leq \exp(-n^c).
  \]
\end{theorem}

\begin{remark}
  Nguyen's theorem, as well as ours, continues to hold in the more general case when the sum of each row is $d$ (instead of $n/2$) for a much wider range of $d$. Here, as in \cite{nguyen2011optimal}, we have chosen to restrict ourselves to the case when $n$ is even and $d=n/2$ for ease of exposition.  
\end{remark}

\subsubsection{Further directions and related work}
\label{sec:further-directions}

The methods we use in this paper can be further developed in various directions. 
In a recent work~\cite{ferber2018singularity}, the first two named authors utilized and extended some of the ideas introduced here in order to provide the best known upper bound for the well studied problem of estimating the singularity probability of random \emph{symmetric} $\{\pm 1\}$-valued matrices, and in upcoming work \cite{jain2019combinatorial}, the second named author uses some of the results in this paper to study the non-asymptotic behavior of the least singular value of different models of discrete random matrices. In another upcoming work of the second named author \cite{jain2019b}, it is shown how to extend the techniques introduced here and in \cite{jain2019combinatorial} to study not-necessarily-discrete models of random matrices. We also anticipate that the techniques presented here (along with some additional combinatorial ideas) should suffice to provide an `exponential-type' upper bound on the probability of singularity of the adjacency matrix of a dense random regular digraph, thereby making substantial progress towards a conjecture of Cook~\cite[Conjecture~1.7]{cook2017singularity}. 

\bigskip
\noindent
\textbf{Organization:} The rest of this paper is organized as follows. \cref{sec:pf-counting-thm} is devoted to the proof of \cref{thm:counting-lemma}. In~\cref{sec:good-bad-vectors}, we formulate and prove abbreviated, easy-to-use versions of \cref{thm:halasz-fp,thm:counting-lemma}. \cref{sec:rrd,sec:singularity-row-reg} are devoted to the proofs of \cref{thm:singularity-rrd,thm:singularity-row-regular}, respectively. We provide detailed proof outlines at the start of both \cref{sec:rrd,sec:singularity-row-reg}. Finally, \cref{app:halasz} contains the proof of Hal\'asz's inequality over $\F_p$ (\cref{thm:halasz-fp}).

\bigskip
\noindent
\textbf{Notation: }Throughout this paper, we will routinely omit floor and ceiling signs when they make no essential difference. As is standard, we will use $[n]$ to denote the discrete interval $\{1,\dots,n\}$. We will also use the asymptotic notation $\lesssim, \gtrsim, \ll, \gg$ to denote $O(\cdot), \Omega(\cdot), o(\cdot), \omega(\cdot)$ respectively. All logarithms are natural unless noted otherwise.

\bigskip
\noindent \textbf{Acknowledgements: } A.F. is partially supported by NSF 6935855, V.J. is partially supported by NSF CCF 1665252, NSF DMS-1737944, and ONR N00014-17-1-2598, K.L. is partially supported by NSF DMS-1702533, and W.S. is partially supported by grants 1147/14 and 1145/18 from the Israel Science Foundation. W.S. would like to thank Elchanan Mossel and the MIT Mathematics Department for their hospitality during a period when part of this work was completed.

\section{Proof of the counting theorem}
\label{sec:pf-counting-thm}

In this section, we prove~\cref{thm:counting-lemma} using an elementary double counting argument.

\begin{proof}[Proof of~\cref{thm:counting-lemma}]
  Let $\cZ$ be the set of all triples
  \[
    \left(I, \left(i_{s+1},\dots,i_{n}\right), \left(F_{j},{\boldsymbol{\epsilon}}^{j}\right)_{j=s+1}^{n} \right),
  \]
  where
  \begin{enumerate}[{label=(\roman*)}]
  \item $I \subseteq [n]$ and $|I|=s$, 
  \item $(i_{s+1},\dotsc,i_n) \in [n]^{n-s}$ is a permutation of $[n]\setminus I$,
  \item each $F_{j}:=(\ell_{j,1},\dotsc,\ell_{j,2k})$ is a sequence of $2k$ elements of $[n]$, and
  \item $\boldsymbol{\epsilon}^j\in \{\pm 1\}^{2k}$ for each $j$,
  \end{enumerate}
  that satisfy the following conditions for each $j$:
  \begin{enumerate}[{label=(\alph*)}]
  \item
    \label{item:Z-condition-1}
    $\ell_{j,2k}= i_{j}$ and
  \item
    \label{item:Z-condition-2}
    $(\ell_{j,1},\dotsc,\ell_{j,2k-1}) \in \big(I\cup \{i_{s+1},\dotsc,i_{j-1}\}\big)^{2k-1}$.
  \end{enumerate}

  \begin{claim}
    The number of triples in $\cZ$ is at most $(s/n)^{2k-1} \cdot \big(2^{n-s} n! / s!\big)^{2k}$.
  \end{claim}
  \begin{proof}
    One can construct any such triple as follows. First, choose an $s$-element subset of $[n]$ to serve as $I$. Second, considering all $j \in \{s+1, \dotsc, n\}$ one by one in increasing order, choose: one of  the $n-j+1$ remaining elements of $[n] \setminus I$ to serve as $i_j$; one of the $2^{2k}$ possible sign patterns to serve as $\boldsymbol{\epsilon}^{j}$; and one of the $(j-1)^{2k-1}$ sequences of $2k-1$ elements of $I \cup \{i_{s+1},\dots,i_{j-1}\}$ to serve as $(\ell_{j,1},\dotsc,\ell_{j,2k-1})$. Therefore,
    \begin{align*}
      |\cZ| & \le \binom{n}{s} \cdot \prod_{j=s+1}^n \left((n-j+1) \cdot 2^{2k} \cdot (j-1)^{2k-1}\right) \\
            & = \frac{n!}{s!(n-s)!} \cdot (n-s)! \cdot 2^{2k(n-s)} \cdot \left(\frac{(n-1)!}{(s-1)!}\right)^{2k-1} = \left(\frac{s}{n}\right)^{2k-1} \cdot \left(2^{n-s} \cdot \frac{n!}{s!}\right)^{2k}.\qedhere
    \end{align*}
  \end{proof}
  
  We call $\boldsymbol{a} = (a_1, \dotsc, a_n) \in \F_p^n$ \emph{compatible} with a triple from $\cZ$ if for every $j \in \{s+1, \dotsc, n\}$,
  \begin{equation}
    \label{eq:compatibility}
    \sum_{i=1}^{2k}\boldsymbol\epsilon^{j}_ia_{\ell_{j,i}}=0.
  \end{equation}
  
  \begin{claim}
    Each triple from $\cZ$ is compatible with at most $p^s$ sequences $\boldsymbol{a} \in \F_p^n$.
  \end{claim}
  \begin{proof}
    Using~\ref{item:Z-condition-1}, we may rewrite~\cref{eq:compatibility} as
    \[
      \boldsymbol{\epsilon}^{j}_{2k}a_{i_{j}} = -\sum_{i=1}^{2k-1}\boldsymbol\epsilon^{j}_ia_{\ell_{j,i}}.
    \]
    It follows from~\ref{item:Z-condition-2} that once a triple from $\cZ$ is fixed, the right-hand side above depends only on those coordinates of the vector $\boldsymbol{a}$ that are indexed by $i \in I \cup \{i_{s+1}, \dotsc, i_{j-1}\}$. In particular, for each of the $p^s$ possible values of $(a_i)_{i \in I}$, there is exactly one way to extend it to a sequence $\boldsymbol{a} \in \F_p^n$ that satisfies~\cref{eq:compatibility} for every $j$.
  \end{proof}

  \begin{claim}
    Each sequence $\boldsymbol{a} \in\Bad_{k,s, \ge t}^\alpha$ is compatible with at least
    \[
     \left(\frac{2^{n-s}n!}{s!}\right)^{2k} \cdot \left(\frac{\alpha t}{p}\right)^{n-s}
    \]
    triples from $\cZ$.
  \end{claim}
  \begin{proof}
    Given any such $\boldsymbol{a}$, we may construct a compatible triple from $\cZ$ as follows. Considering all $j \in \{n, \dotsc, s+1\}$ one by one in decreasing order, we do the following. First, we find an arbitrary solution to
    \begin{equation}
      \label{eq:a-ell-solution}
      \pm a_{\ell_1} \pm a_{\ell_2} \pm \dotsb \pm a_{\ell_{2k}} = 0
    \end{equation}
    such that $\ell_1, \dotsc, \ell_{2k} \in [n]\setminus \{i_{n},\dots,i_{j+1}\}$ and such that $\ell_{2k}$ is a non-repeated index (i.e., such that $\ell_{2k} \neq \ell_i$ for all $i \in [2k-1]$). Given any such solution, we let $\ell_{2k}$ serve as $i_j$, we let the sequence $(\ell_1, \dotsc, \ell_{2k})$ serve as $F_j$, and we let $\boldsymbol{\epsilon}^j$ be the corresponding sequence of signs (so that~\cref{eq:compatibility} holds). The assumption that $\boldsymbol{a} \in \Bad_{k,s,\geq t}^\alpha(n)$ guarantees that there are at least  $t \cdot \frac{2^{2k}\cdot (n-j+1)^{2k}}{p}$ many solutions to~\cref{eq:a-ell-solution}, each of which has at least $2\alpha k$ nonrepeated indices. Since the set of all such solutions is closed under every permutation of the $\ell_i$s (and the respective signs), $\ell_{2k}$ is a non-repeated index in at least an $\alpha$-proportion of them. Finally, we let $I = [n] \setminus \{i_n, \dotsc, i_{s+1}\}$. Since different sequences of solutions lead to different triples, it follows that the number $Z$ of compatible triples satisfies
    \[
      Z \ge \prod_{j = s+1}^{n} \left(\alpha t \cdot \frac{2^{2k} \cdot (n-j+1)^{2k}}{p}\right) = \left(\frac{2^{n-s}n!}{s!}\right)^{2k}\cdot \left(\frac{\alpha t}{p}\right)^{n-s}.\qedhere
    \]
  \end{proof}

  Counting the number $P$ of pairs of $\boldsymbol{a} \in \Bad_{k, s, \ge t}^\alpha(n)$ and a compatible triple from $\cZ$, we have
  \[
    |\Bad_{k, s, \ge t}^\alpha(n)| \cdot \left(\frac{2^{n-s}n!}{s!}\right)^{2k} \cdot \left(\frac{\alpha t}{p}\right)^{n-s}\le P \le |\cZ| \cdot p^s \le \left(\frac{s}{n}\right)^{2k-1} \cdot \left(\frac{2^{n-s}n!}{s!}\right)^{2k} \cdot p^s,
  \]
  which yields the desired upper bound on $|\Bad_{k, s, \ge t}^\alpha(n)|$.
\end{proof}

\section{`Good' and `bad' vectors}
\label{sec:good-bad-vectors}

The purpose of this section is to formulate easy-to-use versions of Hal\'asz's inequality (\cref{thm:halasz-fp}) and our counting theorem (\cref{thm:counting-lemma}). We shall partition $\F_p^*$ -- the set of all finite-dimensional vectors with $\F_p$-coefficients -- into `good' and `bad' vectors. We shall then show that, on the one hand, every `good' vector has small largest atom probability and that, on the other hand, there are relatively few `bad' vectors.\footnote{In fact, we shall only show that there are relatively few `bad' vectors that have some number of nonzero coordinates. The number of remaining vectors (ones with very small support) is so small that even a crude, trivial estimate will suffice for our needs.} The formal statements now follow. In order to simplify the notation, we suppress the implicit dependence of the defined notions on $k$, $p$, and $\alpha$.

\begin{definition}
  \label{def:Ht-rrm}
  Suppose that an integer $k$, a prime number $p$, and an $\alpha \in (0,1)$ are given. For any $t>0$, define the set $H_t$ of \emph{$t$-good} vectors by
  \[
    \boldsymbol{H}_t:= \left\{\boldsymbol{a}\in \F_p^*: \exists \boldsymbol{b}\subseteq \boldsymbol{a} \text{ with }|\supp(\boldsymbol{b})|\geq |\boldsymbol{a}|^{1/4} \text{ and }\Rka(\boldsymbol{b})\leq t\cdot \frac{2^{2k} \cdot |\boldsymbol{b}|^{2k}}{p}\right\}.
  \]
  The \emph{goodness} of a vector $\boldsymbol{a} \in \F_p^*$, denoted by $h(\boldsymbol{a})$, will be the smallest $t$ such that $\boldsymbol{a} \in \boldsymbol{H}_t$. In other words
  \[
    h(\boldsymbol{a}) = \min\left\{\frac{p  \cdot \Rka(\boldsymbol{b})}{2^{2k} \cdot |\boldsymbol{b}|^{2k}} : \boldsymbol{b} \subseteq \boldsymbol{a} \text{ and } |\supp(\boldsymbol{b})| \ge |\boldsymbol{a}|^{1/4}\right\}.
  \]
\end{definition}

Note that if a vector $\boldsymbol{a} \in \F_p^*$ has fewer than $|\boldsymbol{a}|^{1/4}$ nonzero coordinates, then it cannot be $t$-good for any $t$ and thus $h(\boldsymbol{a}) = \infty$. On the other hand, since trivially $\Rka(\boldsymbol{b}) \le 2^{2k} \cdot |\boldsymbol{b}|^{2k}$ for every vector $\boldsymbol{b}$, every $\boldsymbol{a} \in \F_p^*$ with at least $|\boldsymbol{a}|^{1/4}$ nonzero coordinates must be $p$-good, that is, $h(\boldsymbol{a}) \le p$ for each such $\boldsymbol{a}$.

Having formalized the notion of a good vector, we are now ready to state and prove two corollaries of \cref{thm:halasz-fp,thm:counting-lemma} that lie at the heart of our approach to the singularity problem.

\begin{lemma}
  \label{lem:smallballdreg}
  Suppose that $\boldsymbol{a} \in \boldsymbol{H}_t$. If $t \geq |\boldsymbol{a}|^{1/4}$, $k \le |\boldsymbol{a}|^{1/8}$, and $p \le 2^{k/100}$, then
  \[
    \rho_{\F_p}(\boldsymbol{a}) \le \frac{Ct}{p|\boldsymbol{a}|^{1/16}},
  \]
  where $C = C(\alpha)$ is a constant that depends only on $\alpha$.
\end{lemma}
\begin{proof}
  Let $\boldsymbol{a}$ be a finite-dimensional vector with $\F_p$-coefficients and suppose that $\boldsymbol{a} \in \boldsymbol{H}_t$ for some $t \ge |\boldsymbol{a}|$. Denote $|\boldsymbol{a}|$, the dimension of $\boldsymbol{a}$, by $n$. Without loss of generality, we may assume that $n$ is larger than any function of $\alpha$, since otherwise our assumptions imply that the claimed upper bound on $\rho_{\F_p}(\boldsymbol{a})$ is greater than one whenever $C = C(\alpha)$ is sufficiently large. Let $\boldsymbol{b}$ be an arbitrary subvector of $\boldsymbol{a}$ such that $|\supp(\boldsymbol{b})| \ge n^{1/4}$ and $\Rka(\boldsymbol{b}) \le t \cdot 2^{2k} \cdot |\boldsymbol{b}|^{2k}/p$. Set $M = \lfloor n^{1/4} / (80k) \rfloor$ so that
  \[
    \max\{30M, 80Mk\} = 80Mk \le n^{1/4} \le |\supp(\boldsymbol{b})| \le |\boldsymbol{b}|
  \]
  and note that our assumptions imply that $M \ge n^{1/4} / (100k) \ge n^{1/8}/100$. \cref{thm:halasz-fp} and \cref{lemma:R_k vs Rka} give
  \begin{align*}
    \rho_{\F_{p}}(\boldsymbol{b}) & \le \frac{1}{p}+\frac{R_{k}(\boldsymbol{b})}{2^{2k} \cdot |\boldsymbol{b}|^{2k} \cdot M^{1/2}}+e^{-M} \\
                                  & \leq \frac{1}{p}+\frac{R_{k}^{\alpha}(\boldsymbol{b})+\left(40k^{1-\alpha} |\boldsymbol{b}|^{1+\alpha}\right)^k}{2^{2k} \cdot |\boldsymbol{b}|^{2k} \cdot M^{1/2}}+e^{-M}\\
                                  & \leq \frac{1}{p}+\frac{t \cdot 2^{2k} \cdot |\boldsymbol{b}|^{2k}/p+\left(40k^{1-\alpha} |\boldsymbol{b}|^{1+\alpha}\right)^k}{2^{2k} \cdot |\boldsymbol{b}|^{2k} \cdot M^{1/2}}+e^{-M}\\
                                  & = \frac{1}{p} \left(1 + \frac{t}{M^{1/2}} +\left( 10 (k / |\boldsymbol{b}|)^{1-\alpha}\right)^k \cdot \frac{p}{M^{1/2}}\right) + e^{-M}.
  \end{align*}
  Since $p \le 2^{k/100} \le e^{n^{1/4}/(100k)} \le e^M$ and
  \[
    \left( 10 (k / |\boldsymbol{b}|)^{1-\alpha}\right)^k \cdot \frac{p}{M^{1/2}} \le \left( 10 \cdot n^{(\alpha-1)/8} \right)^k \cdot p \le  2^{-k} \cdot p \le 1,
  \]
  as $\alpha < 1$ and $n$ is large, we may conclude that
  \[
    \rho_{\F_p}(\boldsymbol{a}) \le \rho_{\F_p}(\boldsymbol{b}) \le \frac{1}{p} \left(3 + \frac{t}{M^{1/2}}\right) \le \frac{4t}{pM^{1/2}} \le \frac{40t}{pn^{1/16}},
  \]
  where the last two inequalities hold as $t \ge n^{1/4} \ge M^{1/2} \ge n^{1/16}/10$.
\end{proof}

\begin{lemma}
  \label{lemma:counting-bad}
  For every integer $n$ and real $t \ge n$,
  \[
    \left|\left\{\boldsymbol{a} \in \F_p^n : |\supp(\boldsymbol{a})| \ge n^{1/4} \text{ and } \boldsymbol{a} \not\in \boldsymbol{H}_t\right\}\right| \le \left(\frac{2p}{\alpha t}\right)^n \cdot t^{n^{1/4}}.
  \]
\end{lemma}
\begin{proof}
  We may assume that $t \le p$, as otherwise the left-hand side above is zero, see the comment below \cref{def:Ht-rrm}. Let us first fix an $S \subseteq [n]$ with $|S| \ge n^{1/4}$ and count only vectors $\boldsymbol{a}$ with $\supp(\boldsymbol{a}) = S$. Since $\boldsymbol{a} \not\in \boldsymbol{H}_t$, the restriction $\boldsymbol{a}|_S$ of $\boldsymbol{a}$ to the set $S$ must be contained in the set $\boldsymbol{B}_{k,n^{1/4},\geq t}(|S|)$. Hence, \cref{thm:counting-lemma} implies that the number of choices for $\boldsymbol{a}|_S$ is at most
  \[
    \left(\frac{n^{1/4}}{|S|}\right)^{2k-1}\left(\frac{p}{\alpha t}\right)^{|S|}(\alpha t)^{n^{1/4}} \leq \left(\frac{p}{\alpha t}\right)^n t^{n^{1/4}},
  \]
  where the second inequality follows as $\alpha t \le t \le p$. Since $\boldsymbol{a}|_S$ completely determines $\boldsymbol{a}$, we obtain the desired conclusion by summing the above bound over all sets $S$.
\end{proof}

\section{Singularity of signed r.r.d.\ matrices}
\label{sec:rrd}
\subsection{Overview of the proof and preliminary reductions}
\label{sec:rrd-overview}

In order to facilitate the use of \cref{thm:counting-lemma}, we aim to bound from above the probability that $M_{n,d}^{\pm}$ is singular over $\F_p$, for a suitably chosen prime $p$. This is clearly sufficient as an integer matrix that is singular (over $\Q$ or any of its extensions) is also singular over $\F_p$, for every prime $p$.

As a first step, let $\mathcal{S}^{c}$ denote the event that some vector $\boldsymbol{v}\in \F_p^{n}\setminus \{\boldsymbol{0}\}$ with small support (i.e., with at most $n^{0.8}$ nonzero coordinates) satisfies $M^{\pm}_{n,d}\boldsymbol{v}=0$. Using an elementary union bound argument, we will show in \cref{lemma:rrd-no-small-supp} that $\Pr(\mathcal{S}^{c})$ is extremely small. Therefore, it will suffice to bound from above the probability that $M^{\pm}_{n,d}$ is singular and $\mathcal{S}$ occurs.

As in \cite{cook2017singularity}, we will find it more convenient to work with the following representation of signed r.r.d.\ matrices. Let $\M_{n,d}$ denote the set of all $\{0,1\}$-valued $n \times n$ matrices whose each row and each column sums to $d$ and let $\S_n$ denote the set of all $\{\pm 1\}$-valued $n \times n$ matrices. Let $M_{n,d}$ denote a uniformly random element of $\M_{n,d}$ and let $\Xi_{n}$ denote a uniformly random element of $\S_n$, chosen independent of $M_{n,d}$. It is readily observed that (under the obvious coupling)
\[
  M^{\pm}_{n,d} = \Xi_n \circ M_{n,d},
\]
where $\circ$ denotes the Hadamard product of two matrices (so that $M^{\pm}_{n,d}(i,j) = \Xi_n(i,j) \cdot M_{n,d}(i,j)$ for all $i,j\in [n]$). An equivalent way of saying this is that the pushforward measure of the uniform measure on $\M_{n,d} \times \S_n$ under the map $\circ : \M_{n,d} \times \S_n \to \M_{n,d}^{\pm}$ coincides with the uniform measure on~$\M_{n,d}^{\pm}$. We will refer to $M_{n,d}$ as the \emph{base} of the signed r.r.d.\ matrix $\M_{n,d}^{\pm}$. Observe that $M_{n,d}$ can be viewed as the (bi)adjacency matrix of a uniformly random $d$-regular bipartite graph with $n+n$ vertices.

Similarly as in~\cite{cook2017singularity}, we will first condition on a `good' realization of the base matrix $M_{n,d}$ and later use only the randomness of $\Xi_n$. Of course, we will need to show that such `good' realizations of the base matrix occur with high probability. More precisely, we will identify a subset $\EE_{n,d} \subseteq \M_{n,d}$ of base matrices with suitable `expansion' properties and use the following elementary chain of inequalities:
\begin{align*}
  \Pr\left(M^{\pm}_{n,d}\text{ is singular}\cap \mathcal{S}\right)
  &= \Pr\left(\Xi_n \circ M_{n,d} \text{ is singular}\cap \mathcal{S}\right)\\
  &\leq \Pr\big((\Xi_n \circ M_{n,d}\text { is singular}\cap \mathcal{S})\cap (M_{n,d} \in \EE_{n,d}) \big) + \Pr\left(M_{n,d}\not\in \EE_{n,d}\right)\\
  &\leq \sup_{M \in \EE_{n,d}}\Pr\left(\Xi_{n}\circ M\text{ is singular}\cap \mathcal{S}\right) + \Pr\left(M_{n,d}\not\in \EE_{n,d}\right).
\end{align*}
Roughly speaking, the expansion property that makes the adjacency matrix of a bipartite $d$-regular graph $B$ belong to $\EE_{n,d}$ is the following. Denoting the bipartition of $B$ by $V_1 \cup V_2$, we require that for every moderately large (of size at least $n^{0.6}$) subset $S \subseteq V_1$, all but very few (at most $n^{0.6}$) vertices of $V_2$ have at least $d|S|/(2n)$ many neighbors in $S$. As it turns out, this is a fairly weak property in the sense that (with relatively little work) we will be able to give a very strong upper bound on the probability that $M_{n,d} \not\in \EE_{n,d}$; this is done in \cref{lemma:rrd-expanding-base-whp}. We note that the proof in \cite{cook2017singularity} also proceeds in a simliar fashion. However, the expansion properties required there are much stronger than what we require. Therefore, bounding the respective probability (of not having such expansion) requires considerably more work. In fact, for this reason the proof in \cite{cook2017singularity} is not self-contained; it relies on a previous work of the author on random regular graphs~\cite{cook2017discrepancy}.

The main part of our argument is bounding the supremum above. Fix an arbitrary $M \in \EE_{n,d}$, let $M_1,\dotsc, M_n$ denote its rows, let $M_1', \dotsc, M_n'$ denote the (random) rows of $\Xi_n \circ M$, and let 
$$S_i':= \Span\{M_1',\dotsc,M_{i-1}',M_{i+1}',\dotsc,M_n'\}.$$ 
Observe that $\Xi_n \circ M$ is singular if and only if $M_i' \in S_i'$ for some $i \in [n]$. Furthermore, $M_i' \in S_i'$ if and only if $M_i'$ is orthogonal to every vector in the orthogonal complement of $S_i'$ in $\F_p^n$.  Denote by $\boldsymbol{V}$ the set of all vectors in $\F_p^{n}$ whose support is not small (i.e., vectors with more than $n^{0.8}$ nonzero coordinates). The above observations and the definition of $\mathcal{S}$ yield
\begin{align*}
  \Pr\left(\Xi_{n}\circ M \text{ is singular}\cap \mathcal{S}\right)
&\leq \sum_{i=1}^{n}\Pr\left(M_i' \perp \boldsymbol{v} \text{ for all } \boldsymbol{v} \in (S_i')^{\perp}\cap \boldsymbol{V}\right)\\
&= \sum_{i=1}^{n}\E\left[\Pr\left(M_i' \perp \boldsymbol{v} \text{ for all } \boldsymbol{v} \in (S_i')^{\perp}\cap \boldsymbol{V} \mid S_i'\right)\right]\\
&\leq \sum_{i=1}^{n}\E\left[\inf_{\boldsymbol{v}\in (S_i')^{\perp}\cap\boldsymbol{V}}\rho_{\F_p}\left(\boldsymbol{v}\circ {M_i}\right)\right]\\
&\leq n \cdot \max_{i \in [n]} \left(\E\left[\inf_{\boldsymbol{v}\in (S_i')^{\perp}\cap\boldsymbol{V}}\rho_{\F_p}\left(\boldsymbol{v}\circ{M_i}\right)\right]\right).
\end{align*}
Let $i_0 \in [n]$ be an index that attains the maximum in the above expression. For every $\rho > 0$, let $\mathcal{B}_{\rho}$ denote the event that there exists a vector $\boldsymbol{v}\in (S_{i_0}')^{\perp}\cap \boldsymbol{V}$ such that $\rho_{\F_p}(\boldsymbol{v}\circ M_{i_0}) \geq \rho$.\footnote{Even though the event $\mathcal{B}_\rho$ is that \emph{supremum}, rather than the \emph{infimum}, of $\rho_{\F_p}(\boldsymbol{v} \circ M_{i_0})$ is at least $\rho$, in the case of interest, $S_{i_0}'$ has dimension $n-1$ and thus there is only one (up to a scalar multiple) vector $\boldsymbol{v}$ in $(S_{i_0}')^\perp$.} We may conclude that
\[
  \Pr\left(\Xi_{n}\circ M \text{ is singular}\cap\mathcal{S}\right) \leq n \cdot \inf_{\rho > 0} \big(\Pr(\mathcal{B}_{\rho}) + \rho \big).
\]
It remains to bound from above the probability of $\mathcal{B}_\rho$. By the union bound,
\begin{equation}
  \label{eq:B-rho-union-bound}
  \begin{split}
    \Pr(\mathcal{B}_{\rho}) 
    & \leq \sum_{\substack{\boldsymbol{v}\in\boldsymbol{V} \\ \rho_{\F_p}(\boldsymbol{v}\circ M_{i_0})\geq \rho}}\Pr\big(M_i' \cdot \boldsymbol{v} = 0\text{ for all } i \in [n]\setminus \{i_0\}\big)\\
    & \leq  \sum_{\substack{\boldsymbol{v}\in \boldsymbol{V} \\ \rho_{\F_p}(\boldsymbol{v}\circ M_{i_0})\geq \rho}}\prod_{i\in [n]\setminus \{i_0\}}\rho_{\F_p}(\boldsymbol{v}\circ M_i).
  \end{split}
\end{equation}
We will bound the sum on the right-hand side above in two stages. First, we will give an upper bound on
\[
  \sum_{\boldsymbol{v}\in \boldsymbol{V_1}}\prod_{i\in [n]\setminus \{i_0\}}\rho_{\F_p}(\boldsymbol{v}\circ M_i),
\]
where $\boldsymbol{V_1}\subseteq \boldsymbol{V}$ denotes the set of vectors $\boldsymbol{v}\in \boldsymbol{V}$ for which `many' (at least $n^{0.7}$) of the atom probabilities $\rho_{\F_p}(\boldsymbol{v}\circ M_i)$ are `large'. For this, we stratify the set $\boldsymbol{V_1}$, essentially according to the size of $\prod_{i\in [n]\setminus \{i_0\}}\rho_{\F_p}(\boldsymbol{v}\circ M_i)$, and use the corollary of our counting theorem (\cref{lemma:counting-bad}) along with the expansion property of the base matrix to control the number of vectors in each stratum (\cref{lemma:rrd-not-bad-wrt-many}). Therefore, it only remains to bound from above
\[
  \sum_{\substack{\boldsymbol{v}\in \boldsymbol{V} \setminus \boldsymbol{V_1} \\ \rho_{\F_p}(\boldsymbol{v}\circ M_{i_0})\geq \rho}}\prod_{i\in [n]\setminus \{i_0\}}\rho_{\F_p}(\boldsymbol{v}\circ M_i).
\]
We will do this by first using the counting theorem to show that the size of the set $\{\boldsymbol{v}\in \boldsymbol{V}: \rho_{\F_p}(\boldsymbol{v}\circ M_{i_0})\geq \rho\}$ is `small' (\cref{lemma:rrd-number-bad-wrt-1}) and then bounding the product $\prod_{i\in [n]\setminus \{i_0\}}\rho_{\F_p}(\boldsymbol{v}\circ M_i)$ using the fact that $\boldsymbol{v}\notin \boldsymbol{V_1}$ (\cref{prop:rrd-bounding-Brho}).

We present complete details below. As stated earlier, we make no attempt to optimize our bound on the singularity probability. Consequently, we choose various parameters conveniently (but otherwise somewhat arbitrarily) in order to simplify the exposition. Throughout, $r \in (0,1]$ is fixed, $n$ is a sufficiently large integer, $d = \lceil rn \rceil$, and  various implicit constants are allowed to depend on $r$. Moreover, we set $\alpha = 1/2$, $k = n^{1/8}$, and let $p$ be an arbitrary prime satisfying $2^{n^{0.1}}/2 \le p \le 2^{n^{0.1}}$. Note that this choice of parameters makes \cref{lem:smallballdreg} applicable to vectors $\boldsymbol{a} \in \F_p^n \cap \boldsymbol{H}_t$ as long as $t \ge n^{1/4}$.

\subsection{Eliminating potential null vectors with small support}

We will now show that it is very unlikely that a nonzero vector in $\F_p^n$ with small support is a null vector of $M^{\pm}_{n,d}$. This will be useful in order to apply Hal\'asz's inequality (\cref{thm:halasz-fp}) effectively later. In fact, we show an even stronger statement -- the matrix $M_{n,d}^{\pm}$ is very unlikely to have a null vector with small support even after we condition on the base matrix $M_{n,d}$.

\begin{proposition} 
  \label{lemma:rrd-no-small-supp}
  For every $M \in \mathcal{M}_{n,d}$,
  \[
    \Pr\left(\exists \boldsymbol{v}\in \F_p^{n}\setminus\{\boldsymbol{0}\}\text{ such that }(\Xi_{n}\circ M)\boldsymbol{v}=\boldsymbol{0}\text{ and }|\supp(\boldsymbol{v})|\leq n^{0.8} \right) \lesssim 2^{-d/2}.
  \]
\end{proposition}
\begin{proof}
  Fix any nonzero vector $\boldsymbol{v} = (v_1, \dotsc, v_n) \in \F_p^{n}$ and let $j$ be an arbitrary index such that $v_j \neq 0$. Let $i_1,\dotsc, i_d \in [n]$ be distinct indices such that $M(i_k,j) =1$ for all $k\in [d]$; in other words, $i_1,\dotsc,i_d$ are the indices of the $d$ rows of $M$ which have nonzero entries in the $j$th column. We claim that for every $k \in [d]$,
  \[
    \Pr(M_{i_k}' \cdot \boldsymbol{v} = 0) \le 1/2.
  \]
  To see this, simply condition on all the coordinates of $M_{i_k}'$ except for the $j$th, which is equally likely to be $1$ or $-1$. Since $v_j \neq 0$, then at most one of these two outcomes makes $M_{i_k}' \cdot \boldsymbol{v}$ zero. The rows of $\Xi_n$ are independent and thus the probability that $\boldsymbol{v}$ is orthogonal to all of $M_{i_1}',\dotsc,M_{i_d}'$ is at most~$2^{-d}$. Finally, note that the number $N$ of vectors with support of size at most $n^{0.8}$ satisfies
  \[
    N \le \binom{n}{n^{0.8}} \cdot p^{n^{0.8}} \leq n^{n^{0.8}} \cdot p^{n^{0.8}} \leq \left(n \cdot 2^{n^{0.1}}\right)^{n^{0.8}} \lesssim 2^{d/2}
  \]
  and hence the union bound over all such vectors yields the desired conclusion.
\end{proof}

\subsection{Expanding base matrices}

We now formally define the expansion property mentioned in the overview. Let $\EE_{n,d}$ be the set of all matrices $M \in \M_{n,d}$ satisfying the following property for every subset $S\subseteq [n]$ with $|S|\geq n^{0.6}$ (recall that $M_1, \dotsc, M_n$ are the rows of $M$):
\[
  \left|\left\{i \in [n]: \left|\supp(M_i) \cap S \right|\leq \frac{r|S|}{2}\right\}\right| \leq n^{0.6}.
\]
We shall show that it is very unlikely that a uniformly random element of $\M_{n,d}$ is not in $\EE_{n,d}$.  

\begin{proposition} 
  \label{lemma:rrd-expanding-base-whp}
  Let $M_{n,d}$ denote a uniformly random element of $\M_{n,d}$. Then,
  \[
    \Pr\left(M_{n,d} \not\in \EE_{n,d}\right) \lesssim \exp\left(-\frac{rn^{1.2}}{10}\right).
  \]
\end{proposition}
\begin{proof}  
  Let $\widetilde{M}_{n,d}$ denote a random $n \times n$ matrix whose entries are i.i.d.\ $\mathrm{Ber}(d/n)$ random variables and denote its rows by $\widetilde{M}_1, \dotsc, \widetilde{M}_n$. Since each matrix in $\M_{n,d}$ has the same number $nd$ of nonzero entries, then
  \[
    \Pr\left(M_{n,d} \not\in \EE_{n,d}\right) = \Pr\left(\widetilde{M}_{n,d} \not\in \EE_{n,d} \mid \widetilde{M}_{n,d} \in \M_{n,d}\right)
    \leq \frac{\Pr\big(\widetilde{M}_{n,d} \not\in \EE_{n,d}\big)}{\Pr\big(\widetilde{M}_{n,d} \in \M_{n,d}\big)}
  \]
  It was proved\footnote{For the range of parameters we are interested in, such a bound may also be obtained directly using the known lower bound on the number of perfect matchings in a regular bipartite graph, following from the resolution of Van der Waerden's conjecture (see, e.g. \cite{egorychev1981solution}).} by Canfield and McKay \cite[Theorem~1]{canfield2005asymptotic} that 
  \[
    \Pr\left(\widetilde{M}_{n,d} \in \M_{n,d}\right) = \exp\left(-O\big(n \log(\min\{d, n-d\})\big)\right),
  \]
  provided that $\min\{d,n-d\} = \omega\left(n/\log{n}\right)$, so it suffices to bound $\Pr(\widetilde{M}_{n,d} \not\in \EE_{n,d})$ from above. For this, fix any $S \subseteq [n]$ with $|S| = s \geq n^{0.6}$. Since for any $i \in [n]$, the cardinality of $\supp(\widetilde{M}_i) \cap S$ has binomial distribution with mean $ds/n \ge rs$, it follows from standard tail estimates for binomial distributions that 
  \[
    \Pr\left(|\supp(\widetilde{M}_i) \cap S|\leq \frac{rs}{2}\right) \leq \exp\left(-\frac{rs}{8}\right).
  \]
  Since the rows of $\widetilde{M}_{n,d}$ are independent, the probability that there are at least $n^{0.6}$ such indices $i\in [n]$ is at most 
  \[
    \binom{n}{n^{0.6}} \cdot \exp\left(-\frac{rs}{8}\right)^{n^{0.6}} \lesssim \exp\left(n^{0.6}\log n-\frac{rn^{1.2}}{8}\right) \lesssim \exp\left(-\frac{rn^{1.2}}{9}\right).
  \]
  Taking the union bound over all sets $S\subseteq[n]$ with $|S|\geq n^{0.6}$ gives the desired conclusion. 
\end{proof}

\subsection{Bounding $\Pr(\mathcal{B}_\rho)$ for small $\rho$}

Throughout this subsection, we will consider a fixed $M \in \EE_{n,d}$ and denote its rows by $M_1, \dotsc, M_n$. Recall from the proof outline that
\[
  \boldsymbol{V} = \big\{\boldsymbol{v} \in \F_p^n : |\supp(\boldsymbol{v})| > n^{0.8}\big\}
\]
and that $i_0$ is an index that attains the maximum in
\[
  \max_{i \in [n]} \left(\E\left[\inf_{\boldsymbol{v}\in (S_i')^{\perp}\cap\boldsymbol{V}}\rho_{\F_p}\left(\boldsymbol{v}\circ{M_i}\right)\right]\right).
\]
Note that for any $\boldsymbol{v} \in \boldsymbol{V}$, the definition of expanding base matrices yields a subset $T_{\boldsymbol{v}}\subseteq[n]\setminus \{i_0\}$ with $|T_{\boldsymbol{v}}| \geq n-n^{0.6}-1$ such that for each $i \in T_{\boldsymbol{v}}$,
\[
  |\supp(\boldsymbol{v} \circ M_i)| = |\supp(\boldsymbol{v}) \cap \supp(M_i)| \geq \frac{r|\supp(\boldsymbol{v})|}{2} \ge \frac{rn^{0.8}}{2}.
\]
Recall the definitions of the set $\boldsymbol{H}_t$ of $t$-good vectors and of the goodness function $h$ given in \cref{sec:good-bad-vectors}.

\begin{definition}
  \label{def:Btell-rrd}
  For any $t>0$ and $\ell \in \N$, define the set $\boldsymbol{B}_{t,\ell}$ of \emph{$(t,\ell)$-bad} vectors by
  \[
    \boldsymbol{B}_{t,\ell}:= \big\{\boldsymbol{v}\in \boldsymbol{V} : |\{i \in T_{\boldsymbol{v}} : \boldsymbol{v} \circ M_i \notin \boldsymbol{H}_{t}\}| \ge \ell \big\}.
  \]
  We say that a sequence $(i_1, \dotsc, i_\ell)$ of distinct elements of of $T_{\boldsymbol{v}}$ \emph{witnesses} $\boldsymbol{v} \in \boldsymbol{B}_{t, \ell}$ if
  \[
    h(\boldsymbol{v} \circ M_{i_1}) \geq \dotsb \geq h(\boldsymbol{v} \circ M_{i_\ell}) \geq \max_{i \in T_{\boldsymbol{v}} \setminus \{i_0, i_1, \dotsc, i_\ell\}} h(\boldsymbol{v} \circ M_i).
  \]
\end{definition}

Recall from the proof outline that our goal is to bound from above the probability of the event $\mathcal{B}_\rho$ for some very small $\rho$ and that we are planning to do it by splitting the sum in the right-hand side of~\cref{eq:B-rho-union-bound} into two parts, depending on whether or not the vector $\boldsymbol{v} \in \boldsymbol{V}$ admits `many' indices $i \in [n]$ for which the largest atom probability of $\boldsymbol{v} \circ M_i$ is `large'. More precisely, we shall let
\[
  \rho = p^{-1/2}, \qquad \ell = n^{0.7}, \qquad \text{and} \qquad \boldsymbol{V_1} = \boldsymbol{B}_{n, \ell}.
\]
In other words, we first consider those vectors $\boldsymbol{v} \in \boldsymbol{V}$ for which there are at least $n^{0.7}$ indices $i\in [n]$ such that $\boldsymbol{v} \circ M_i$ has large support but nevertheless \cref{lem:smallballdreg} does not give a strong upper bound on $\rho_{\F_p}(\boldsymbol{v} \circ M_i)$.

\begin{lemma} 
\label{lemma:rrd-not-bad-wrt-many}
If $t \ge n$, then
\[
  \sum_{\boldsymbol{v}\in \boldsymbol{B}_{t, \ell}}\prod_{i\in[n]\setminus \{i_0\}}\rho_{\F_p}(\boldsymbol{v} \circ M_i) \lesssim n^{-n/20}.
\]
\end{lemma}
\begin{proof}
  It is enough to show that the contribution to the above sum of $\boldsymbol{v} \in \boldsymbol{B}_{t, \ell}$ that are witnessed by a given sequence $i_1, \dotsc, i_\ell$ of distinct indices in $[n] \setminus \{i_0\}$ and that satisfy $\supp(\boldsymbol{v}) = S$ for a given set $S$ with $|S| > n^{0.8}$ is $O(2^{-n} \cdot n^{-n/20-\ell})$. Indeed, we can then take the union bound over all such sequences and all such sets $S$. Let us then fix such a sequence and a set $S$ for the remainder of the proof.

  We first claim that there are distinct indices $j_1, \dotsc, j_b \in \{i_1, \dotsc, i_\ell\}$ and pairwise disjoint subsets $J_1 \subseteq S \cap \supp(M_{j_1}), \dotsc, J_b \subseteq S \cap \supp(M_{j_b})$ such that $b \leq (2/r) \cdot \log n$ and
  \begin{itemize}
  \item
    $|J_a| \ge n^{0.7}$ for every $a \in [m]$, and
  \item
    $|J_1| + \dotsb + |J_b| = |J_1 \cup \dotsb \cup J_b| \ge |S| - n^{0.75}$.
  \end{itemize}
  Indeed, one may construct these two sequences as follows. Let $I_0 = S$ and for $a = 0, 1, 2, \dotsc$, do the following. If $|I_a| > n^{0.75} $, then the assumption that $M \in \EE_{n,d}$ implies that for all but at most $n^{0.6}$ indices $i \in [n]$, we have $|\supp(M_i) \cap I_a| \ge r|I_a|/2 \ge n^{0.7}$. Since $\ell - a \ge n^{0.7} - (2/r) \cdot \log n > n^{0.6}$, then we can find one such index among $\{i_1, \dotsc, i_\ell\} \setminus \{j_1, \dotsc, j_a\}$; denote this index by $j_{a+1}$, let $J_{a+1} = \supp(M_{j_{a+1}}) \cap I_a$, and let $I_{a+1} = I_a \setminus J_{a+1}$. Otherwise, if $|I_a| \le n^{0.75}$, then let $b = a$ and terminate the process. Since $|I_{a+1}| < (1-r/2) \cdot |I_a|$ for every $a < b$, then $b \le (2/r) \log n$.

  Now, given an integer $m$, let $\boldsymbol{C}_m$ be the set of all vectors $\boldsymbol{v} \in \boldsymbol{B}_{t, \ell}$ with $\supp(\boldsymbol{v}) = S$ that are witnessed by our sequence and for which $2^mt < h(\boldsymbol{v} \circ M_{i_\ell}) \le 2^{m+1}t$. Since $h(\boldsymbol{a}) \le p$ for every vector $\boldsymbol{a}$ with $|\supp(\boldsymbol{a})| \ge |\boldsymbol{a}|^{1/4}$, then the set $\boldsymbol{C}_m$ is empty unless $t \le 2^m t \le p$ and hence $0 \le m \le \log_2 p \le n^{0.1}$.

  Fix any such $m$ and suppose that $\boldsymbol{v} \in \boldsymbol{C}_m$. Since $\{j_1, \dotsc, j_b\} \subseteq \{i_1, \dotsc, i_\ell\}$, it follows from the definition of a witnessing sequence that for every $a \in [b]$,
  \[
    h(\boldsymbol{v}|_{J_a}) \ge h(\boldsymbol{v} \circ M_{j_a}) \ge h(\boldsymbol{v} \circ M_{i_\ell}) > 2^mt,
  \]
  where $\boldsymbol{v}|_{J_a}$ is the restriction of $\boldsymbol{v}$ to the subset $J_a$ of its coordinates and the first inequality is due to the fact that $J_a \subseteq \supp(M_{j_a})$. In particular, $\boldsymbol{v}|_{J_a} \not\in \boldsymbol{H}_{2^m t}$ for every $a \in [b]$. However, since $J_a \subseteq S = \supp(\boldsymbol{v})$, then $|\supp(\boldsymbol{v}|_{J_a})| = |J_a| \ge n^{0.7} \ge |J_a|^{1/4}$, it follows from~\cref{lemma:counting-bad} that for every $a \in [b]$, there are at most $(4p)^{|J_j|} (2^mt)^{-|J_j|+|J_j|^{1/4}}$ possible values of $\boldsymbol{v}|_{J_a}$. Since $J_1, \dotsc, J_b$ are pariwise disjoint subsets of $S$ that cover all but at most $n^{0.75}$ of its elements and $\boldsymbol{C}_m$ contains only vectors whose support is $S$, we may conclude that
  \[
    |\boldsymbol{C}_m| \le p^{n^{0.75}} \cdot \prod_{a=1}^b \left(\frac{4p}{2^mt}\right)^{|J_j|} \cdot (2^mt)^{bn^{1/4}} \lesssim \left(\frac{4p}{2^mt}\right)^{|S|} \cdot \left(2^mtp\right)^{n^{0.75}} \lesssim \left(\frac{4p}{2^mt}\right)^n \cdot p^{2n^{0.75}}.
  \]

  On the other hand, it follows from the definition of a witnessing sequence that for each $\boldsymbol{v} \in \boldsymbol{C}_m$ and every $i \in T_{\boldsymbol{v}} \setminus \{i_0, i_1, \dotsc, i_\ell\}$, we have $h(\boldsymbol{v} \circ M_i) \le h(\boldsymbol{v} \circ M_{i_\ell}) \le 2^{m+1}t$ and thus, by \cref{lem:smallballdreg}, $\rho_{\F_p}(\boldsymbol{v} \circ M_i) \leq 2^{m+1}Ct/(pn^{1/16})$ for some absolute constant $C$. Consequently, every $\boldsymbol{v} \in \boldsymbol{C}_m$ satisfies
  \[
    \prod_{i \in [n] \setminus \{i_0\}} \rho_{\F_p}(\boldsymbol{v} \circ M_i) \le \prod_{i \in T_{\boldsymbol{v}} \setminus \{i_0, i_1, \dotsc, i_\ell\}} \rho_{\F_p}(\boldsymbol{v} \circ M_i) \le \left(\frac{2^{m+1}Ct}{pn^{1/16}}\right)^{|T_{\boldsymbol{v}}| - \ell} \le \left(\frac{2^{m+1}Ct}{pn^{1/16}}\right)^{n-2n^{0.7}}.
  \]

  Putting everything together, we see that
  \begin{align*}
    \sum_{\boldsymbol{v} \in \boldsymbol{C}_m} \prod_{i \in [n]\setminus \{i_0\}}\rho_{\F_p}(\boldsymbol{v} \circ M_i) & \lesssim \left(\frac{4p}{2^mt}\right)^n \cdot p^{2n^{0.75}} \cdot \left(\frac{2^{m+1}Ct}{pn^{1/16}}\right)^{n-2n^{0.7}} \\
                                                                                                                & \lesssim (8C)^n \cdot p^{4n^{0.75}} \cdot n^{-n/17} \\
                                                                                                                & \lesssim (8C)^n \cdot 2^{4n^{0.85}} \cdot n^{-n/17} \lesssim n^{-n/18},
  \end{align*}
  where the penultimate inequality holds because $p \le 2^{n^{0.1}}$. Since there are at most $n^{0.1} + 1$ relevant values of $m$, at most $n^{\ell}$ sequences $i_1, \dotsc, i_\ell$, and at most $2^n$ sets $S$, the claimed upper bound follows.
\end{proof}

The next lemma bounds the number of vectors $\boldsymbol{v}$ for which $\rho(\boldsymbol{v} \circ M_{i_0})$ has large atom probability. 
\begin{lemma} 
\label{lemma:rrd-number-bad-wrt-1}
The number of vectors $\boldsymbol{v}\in \F_p^{n}$ for which $\rho_{\F_p}(\boldsymbol{v} \circ M_{i_0}) \ge \rho$ is $O(p^{n-rn/4})$.
\end{lemma}
\begin{proof}
  We partition the set of relevant vectors $\boldsymbol{v}$ into two parts depending on the size of the support of $\boldsymbol{v} \circ M_{i_0}$. More precisely, we let
  \begin{align*}
    \boldsymbol{V_{\text{small}}} & := \{\boldsymbol{v} \in \F_p^n : |\supp(\boldsymbol{v} \circ M_{i_0})| \le n^{0.8}\}, \\
    \boldsymbol{V_{\text{large}}} & := \{\boldsymbol{v} \in \F_p^n : |\supp(\boldsymbol{v} \circ M_{i_0})| > n^{0.8} \text{ and } \rho_{\F_p}(\boldsymbol{v} \circ M_{i_0}) \ge \rho\}.
  \end{align*}
  Since $M_{i_{0}}$ is a fixed vector with exactly $d$ nonzero entries, then
  \begin{align*}
    |\boldsymbol{V_{\text{small}}}| \le \binom{d}{n^{0.8}} \cdot p^{n-d+n^{0.8}} \lesssim 2^d \cdot p^{n-d+n^{0.8}} \lesssim p^{n-rn/2},
  \end{align*}
  as $d \ge rn \gg n^{0.8}$ and $p \gg 1$. Observe that if $\boldsymbol{v} \in \boldsymbol{V_{\text{large}}}$, then $\boldsymbol{v} \circ M_{i_0} \not\in \boldsymbol{H}_{\rho p}$, as otherwise \cref{lem:smallballdreg} would imply that $\rho_{\F_p}(\boldsymbol{v} \circ M_{i_0}) \leq C\rho/n^{1/16}$ for some absolute constant $C$, contradicting the assumption that $\boldsymbol{v} \in \boldsymbol{V_{\text{large}}}$. In particular, \cref{lemma:counting-bad} implies that there are at most $(4p)^d(\rho p)^{-d+n^{1/4}}$ different restrictions of $\boldsymbol{v} \in \boldsymbol{V_{\text{large}}}$ to $\supp(M_{i_0})$. Recalling that $\rho = p^{-1/2}$, we obtain
  \[
    |\boldsymbol{V_{\text{large}}}| \le p^{n-d} \cdot (4p)^d \cdot (\rho p)^{-d+n^{1/4}} = 4^d \cdot p^{n-d/2+n^{1/4}/2} \lesssim p^{n - rn/3},
  \]
  where the last inequality holds as $rn \le d \le n$ and $p \gg 1$. We obtain the desired conclusion by summing the obtained upper bounds on $|\boldsymbol{V_{\text{small}}}|$ and $|\boldsymbol{V_{\text{large}}}|$.
\end{proof}

We now combine \cref{lem:smallballdreg,lemma:rrd-not-bad-wrt-many,lemma:rrd-number-bad-wrt-1} to derive the main result of this subsection.

\begin{proposition}
  \label{prop:rrd-bounding-Brho}
  We have
  \[
    \Pr(\mathcal{B}_{\rho}) \lesssim n^{-n/20} + p^{-rn/5}.
  \]
\end{proposition}
\begin{proof}
  Recall from~\cref{eq:B-rho-union-bound} and our definition $\boldsymbol{V_1} = \boldsymbol{B}_{n, \ell}$ that
  \begin{equation}
    \label{eqn:rrd-boundB_rho}
    \Pr(\mathcal{B}_{\rho}) \le \sum_{\boldsymbol{v}\in \boldsymbol{B}_{n,\ell}}\prod_{i\in[n]\setminus \{i_0\}}\rho_{\F_{p}}(\boldsymbol{v} \circ M_i)+\sum_{\substack{\boldsymbol{v} \in\boldsymbol{V} \setminus \boldsymbol{B}_{n,\ell} \\ \rho_{\F_p}(\boldsymbol{v} \circ M_{i_0})\geq \rho}} \prod_{i\in[n]\setminus \{i_0\}}\rho_{\F_p}(\boldsymbol{v} \circ M_i).
  \end{equation}

  \cref{lemma:rrd-not-bad-wrt-many} states that the first term in the right-hand side of \cref{eqn:rrd-boundB_rho} is  $O(n^{-n/20})$. In order to bound the second term, note that for any $\boldsymbol{v} \in \boldsymbol{V} \setminus \boldsymbol{B}_{n,\ell}$, there are at least $|T_{\boldsymbol{v}}|-\ell \geq n-2n^{0.7}$ indices $i \in [n]\setminus \{i_0\}$ for which $\boldsymbol{v} \circ M_i \in \boldsymbol{H}_{n}$. \cref{lem:smallballdreg} implies that $\rho_{\F_p}(\boldsymbol{v} \circ M_i) \leq n/p$ for each such index. In particular, if $\boldsymbol{v} \in \boldsymbol{V} \setminus \boldsymbol{B}_{n,\ell}$, then
  \[
    \prod_{i\in[n] \setminus \{i_0\}}\rho_{\F_p}(\boldsymbol{v} \circ M_i) \leq (n/p)^{n-2n^{0.7}}.
  \]
  Since \cref{lemma:rrd-number-bad-wrt-1} implies that
  \[
    |\{\boldsymbol{v}\in\boldsymbol{V}:\rho_{\F_{p}}(\boldsymbol{v} \circ M_{i_0})\geq \rho\}|\lesssim p^{n-rn/4},
  \]
  it follows that the second term in \cref{eqn:rrd-boundB_rho} is bounded from above by 
  \[
    p^{n-rn/4} \cdot (n/p)^{n-2n^{0.7}} \lesssim p^{-rn/4} \cdot p^{2n^{0.7}} \cdot n^n \lesssim p^{-rn/5},
  \]
  where the last inequality follows as $p \gg n$.
\end{proof}

\subsection{Proof of \cref{thm:singularity-rrd}}
The main result of this section is now immediate. 
\begin{proof} [Proof of \cref{thm:singularity-rrd}]
  Recall from \cref{sec:rrd-overview} that for every positive $\rho$,
  \begin{equation}
    \label{eqn:bound-rrd-sing}
    \Pr(M^{\pm}_{n,d}\text{ is singular}) \leq \Pr(\mathcal{S}^c) + \Pr(M_{n,d} \not\in \mathcal{E}_{n,d}) + n \cdot \big( \Pr(\mathcal{B}_{\rho}) + \rho\big).
  \end{equation}
  We know from \cref{lemma:rrd-no-small-supp} that $\Pr(\mathcal{S}^{c}) \lesssim 2^{-d/2}$, from \cref{lemma:rrd-expanding-base-whp} that $\Pr(M_{n,d}\in \EE_{n,d}^{c}) \lesssim 2^{-rn^{1.2}/10} $, and from \cref{prop:rrd-bounding-Brho} that $\Pr(B_\rho) \lesssim n^{-n/20} + p^{-rn/5}$. Thus the dominant term in \cref{eqn:bound-rrd-sing} is $n\rho$. Recalling that $\rho = p^{-1/2}$ and $p \ge 2^{n^{0.1}}/2$, we conclude that the right-hand side of \cref{eqn:bound-rrd-sing} can be bounded from above by $Cn \cdot 2^{-n^{0.1}/2}$ for some absolute constant $C$. This gives the desired conclusion.  
\end{proof}

\section{Singularity of random row-regular matrices: proof of \cref{thm:singularity-row-regular}}
\label{sec:singularity-row-reg}

\subsection{Overview of the proof and preliminary reductions}
\label{sec:overview-row-reg}

The proof of \cref{thm:singularity-row-regular} is very similar to the proof of \cref{thm:singularity-rrd}, as will be clear from the following overview. Throughout this section, we will assume that $n$ is even. Recall that $\QQ_n$ denotes the set of all $n\times n$ matrices with entries in $\{0,1\}$ each of whose rows sums to $n/2$. We will prove the stronger statement that a uniformly chosen random matrix $Q_n \in \QQ_n$ is non-singular even over $\F_p$, for a suitably chosen prime $p$, with extremely high probability.

As a first step, let $\mathcal{S}^{c}$ denote the event that some `almost constant' vector $\boldsymbol{v} \in \F_p^{n}\setminus\{\boldsymbol{0}\}$, i.e., a vector almost all of whose coordinates (all but at most $n^{0.8}$) have the same value, satisfies $Q_{n}\boldsymbol{v} = 0$. More precisely, for a vector $\boldsymbol{v} \in \F_p^n$, we define
\[
  L(\boldsymbol{v}) = \max_{x \in \F_p} |\{i \in [n] : v_i = x\}|
\]
and let $\mathcal{S}^c$ be the event that $Q_n \boldsymbol{v}  = 0$ for some nonzero $\boldsymbol{v}$ with $L(\boldsymbol{v}) \ge n - n^{0.8}$. We will show that $\Pr(\mathcal{S}^{c})$ is extremely small (\cref{prop:eliminate-large-levelset}), so that it will suffice to bound $\Pr(Q_n\text{ is singular}\cap \mathcal{S})$ from above.

As in the previous proof, we will find it more convenient (as will be explained later in this subsection) to work with the following representation of a uniformly random element of $\QQ_n$. Let $\Sigma_{n}$ denote the set of all permutations of $[n]$ and consider the map
\[
  f \colon (\Sigma_{n})^{n} \times \left(\{0,1\}^{n/2}\right)^{n} \to \QQ_n,
\]
which takes $\big((\sigma_1,\dots,\sigma_n), \xi_1,\dots,\xi_n\big)$ to the matrix in $\QQ_n$ whose $i^{th}$ row is $(q_{i1}, \dotsc, q_{in})$, where
\[
  q_{ij} =
  \begin{cases}
    \xi_i(k) & \text{if $\sigma_i(2k-1) = j$},\\
    1 - \xi_i(k) & \text{if $\sigma_i(2k) = j$}.
  \end{cases}
\]
In other words, for each $k \in [n/2]$, exactly one among the $\sigma_i(2k-1)^{st}$ and the $\sigma_i(2k)^{th}$ entries in the $i$th row is equal to $1$ (the other is equal to $0$) and the value of $\xi_i(k)$ determines which one of the two entries it is. It is straightforward to see that the pushforward measure of the uniform measure on $(\Sigma_n)^{n} \times \left(\{0,1\}^{n/2}\right)^{n}$ under the map $f$ is the uniform measure on $\QQ_n$. In other words, the following process generates a uniformly random element of $\QQ_n$. First, choose a sequence $\boldsymbol{\sigma} = (\sigma_1,\dots,\sigma_n)$ of i.i.d.\ uniformly random elements of $\Sigma_n$. Second, for each $i \in [n]$ and each $k \in [n/2]$, choose exactly one among the $\sigma_i(2k-1)^{st}$ entry and the $\sigma_i(2k)^{th}$ entry in the $i^{th}$ row of the matrix to be $1$ (and the other to be $0$) uniformly at random, independently for each pair of indices $i$ and $k$. We shall refer to $\boldsymbol{\sigma}$ as the \emph{base} of the matrix $Q_n$. Let us note here that for each $i \in [n]$, the set comprising the $n/2$ unordered pairs $\{\sigma_i(2k-1), \sigma_i(2k)\}$, for all $k \in [n/2]$, is a uniformly random perfect matching in $K_n$ -- the complete graph on the vertex set $[n]$; we shall refer to this matching as the matching induced by $\sigma_i$.

In analogy with the signed r.r.d.\ case, we will first condition on a `good' realization of the base $\boldsymbol{\sigma}$ and later use only the randomness of $\boldsymbol{\xi}:= (\xi_1,\dotsc,\xi_n)$. More precisely, we will identify a subset $\EE_n \subseteq \left(\Sigma_{n}\right)^{n}$ of bases with suitable `expansion' properties and use the following chain of inequalities. Denote by $Q_{\boldsymbol{\sigma}}$ the random matrix chosen uniformly among all the matrices in $\QQ_n$ with base $\boldsymbol{\sigma}$ and by $\boldsymbol{\tau} \in (\Sigma_n)^n$ the vector of i.i.d.\ uniformly random permutations. Then,
\begin{align*}
  \Pr(Q_n\text{ is singular}\cap\mathcal{S}) & = \Pr(Q_{\boldsymbol{\tau}} \text{ is singular} \cap \mathcal{S}) \\
                                               &\leq \Pr\big(Q_{\boldsymbol{\tau}} \text { is singular}\cap \mathcal{S} \cap (\boldsymbol{\tau} \in \EE_n) \big) + \Pr\left(\boldsymbol{\tau} \not\in \EE_n\right)\\
                                               &\leq \sup_{\boldsymbol{\sigma} \in \EE_n} \Pr(Q_{\boldsymbol{\sigma}} \text { is singular}\cap \mathcal{S}) + \Pr\left(\boldsymbol{\tau} \not\in \EE_n\right).
\end{align*}
Roughly speaking, a base $\boldsymbol{\sigma}$ belongs to $\EE_n$ if the following two conditions are met: for every pair of distinct $i, j \in [n]$, the union of the perfect matchings induced by $\sigma_i$ and $\sigma_j$ has relatively few (at most $n^{0.6}$) connected components; and for every pair $A, B \subseteq [n]$ of disjoint sets, each of which is somewhat large (of size at least $n^{0.8}$), the matching induced by almost every $\sigma_i$ (all but at most $\sqrt{n}/2$) contains many edges with one endpoint in each of $A$ and $B$. As before, it will turn out that these `expansion' properties we require from the base permutation are fairly mild and can easily be proved to hold with very high probability (in \cref{prop:expanding-base-perm-whp}) using two somewhat ad hoc large deviation inequalities (\cref{lemma:random-matching,lemma:cycles-perm}).

In analogy with the signed r.r.d.\ case, the main part of the argument is bounding the supremum above. Fix a $\boldsymbol{\sigma} \in \EE_n$, denote the (random) rows of $Q_{\boldsymbol{\sigma}}$ by $W_1, \dotsc, W_n$, and let $S_i = \Span\{W_1, \dotsc, W_{i-1}, W_{i+1}, \dotsc, W_n\}$. Moreover, denote by $\boldsymbol{V}$ the set of all vectors in $\boldsymbol{v} \in \F_p^n$ with $L(\boldsymbol{v}) < n - n^{0.8}$. An elementary reasoning analogous to the one we used in the signed r.r.d.\ case shows that
\begin{align*}
  \Pr\left(Q_{\boldsymbol{\sigma}}\text{ is singular}\cap \mathcal{S}\right)
  & \leq \sum_{i=1}^{n}\Pr\left(W_i\cdot \boldsymbol{v} = 0 \text{ for all }\boldsymbol{v} \in S_i^{\perp}\cap \boldsymbol{V}\right)\\
  & \leq n \cdot \max_{i\in [n]}\left(\E \left[\inf_{\boldsymbol{v} \in S_i^{\perp}\cap \boldsymbol{V}}\Pr\left(W_i\cdot \boldsymbol{v} =0 \mid S_i \right) \right]\right).
\end{align*}
Let $i_0 \in [n]$ be an index that attains the maximum in the above expression. In order to define an analogue of the event $\mathcal{B}_\rho$ from the previous section, we need to take a little detour and explain how we will bound from above the probability that $W_i \cdot \boldsymbol{v} = 0$.

Since the entries of the random vector $W_i$ are not independent, we cannot use standard anti-concentration techniques directly. However, we may rewrite $W_i \cdot \boldsymbol{v}$ as follows:
\[
  W_i \cdot \boldsymbol{v} = \sum_{k=0}^{n/2} \frac{v_{\sigma_i(2k-1)} + v_{\sigma_i(2k)}}{2} + \sum_{k=0}^{n/2}\left(1-2\xi_i(k)\right) \frac{v_{\sigma_i(2k-1)}  -  v_{\sigma_i(2k)}}{2}.
\]
Since $(1-2\xi_i(1)), \dotsc, (1-2\xi_i(n/2))$ are i.i.d.\ Rademacher random variables independent of $\sigma_i$, then, letting $\boldsymbol{v}_{\sigma_i} \in \F_p^{n/2}$ be the vector whose $k^{th}$ coordinate is $(v_{\sigma_i(2k-1)}-v_{\sigma_i(2k)})/2$, we see that
\begin{equation}
  \label{eq:pairanticoncentration}
  \sup_{x\in \F_p}\Pr(W_i \cdot \boldsymbol{v} = x  \mid \sigma_i) \leq \rho_{\F_p}(\boldsymbol{v}_{\sigma_i}).
\end{equation}

For every $\rho > 0$, let $\mathcal{B}_{\rho}$ be the event that there exists a vector $\boldsymbol{v} \in S_{i_0}^{\perp}\cap \boldsymbol{V}$ such that $\rho_{\F_p}(\boldsymbol{v}_{\sigma_{i_0}})\geq \rho$. We may conclude that
\[
  \Pr\left(Q_{\boldsymbol{\sigma}}\text{ is singular}\cap\mathcal{S}\right) \leq n \cdot \inf_{\rho > 0} \big(\Pr\left(\mathcal{B}_{\rho}\right) + \rho \big).
\]
It remains to bound $\Pr(\mathcal{B}_\rho)$ from above. By the union bound, 
\begin{align*}
  \Pr(\mathcal{B}_{\rho}) 
  & \leq \sum_{\substack{\boldsymbol{v}\in\boldsymbol{V} \\ \rho_{\F_p}(\boldsymbol{v}_{\sigma_{i_0}})\geq \rho}}\Pr\big(W_i \cdot \boldsymbol{v} = 0\text{ for all } i \in [n] \setminus \{i_0\}\big)\\
  & \leq  \sum_{\substack{\boldsymbol{v}\in \boldsymbol{V} \\\rho_{\F_p}(\boldsymbol{v}_{\sigma_{i_0}})\geq \rho}}\prod_{i\in [n]\setminus\{i_0\}}\rho_{\F_p}(\boldsymbol{v}_{\sigma_i}).
\end{align*}
As before, we will control the sum on the right-hand side above in two stages. First, we will bound from above the sum over those vectors $\boldsymbol{v}$ for which two of the values $\rho_{\F_p}(\boldsymbol{v}_{{\sigma}_{i}})$, among a set $T_{\boldsymbol{v}}$ of typical indices $i$, are large; we term such vectors $\boldsymbol{v}$ `bad'. For this, we stratify the set of bad vectors, essentially according to the order of magnitude of $\prod_{i \in [n] \setminus \{i_0\}} \rho_{\F_p}(\boldsymbol{v}_i)$, and use the corollary of our counting theorem (\cref{lemma:counting-bad}) along with the expansion property of the base $\boldsymbol{v}$ to control the number of vectors in each stratum (\cref{lem:not-bad-wrt-2}). Later, we will control the sum over the remaining vectors (\cref{lemma:few-bad-wrt-1,prop:rrm-bounding-Brho}).

We present complete details below. As stated earlier, we make no attempt to optimize the constant $c$ in our bound on the singularity probability. Consequently, we choose various parameters conveniently (but otherwise somewhat arbitrarily) in order to simplify the exposition. Throughout, $n$ is a sufficiently large even integer, $\alpha = 1/2$, $k = n^{1/8}/2$, and $p$ is an arbitrary prime satisfying $2^{n^{0.1}}/2 \le p \le 2^{n^{0.1}}$. Note that this choice of parameters makes \cref{lem:smallballdreg} applicable to vectors $\boldsymbol{a} \in \F_p^{n/2} \cap \boldsymbol{H}_t$ as long as $t \ge n^{1/4}$.

\subsection{Two large deviation inequalities}
\label{sec:two-large-deviation}

In this section, we derive large deviation inequalities for two simple functions of a uniformly random perfect matching of $K_n$ (recall that we have assumed that $n$ is even).

\begin{lemma}
  \label{lemma:random-matching}
  Suppose that $A$ and $B$ are two disjoint subsets of $[n]$ and let $M$ be a uniformly random perfect matching in $K_n$. Then
  \[
    \Pr\left(\left|\big\{\{u,v\} \in M : u \in A \text{ and } v \in B\big\}\right| \le \frac{|A||B|}{8n}\right) \le \exp\left(-\frac{|A||B|}{32n}\right).
  \]
\end{lemma}
\begin{proof}
  Without loss of generality, we may assume that $|A| \le |B|$. Consider the following procedure for generating $M$ one edge at a time. Start with $M_0$ being the empty matching and do the following for $i = 1, \dotsc, \lceil |A|/2 \rceil$. First, let $u_i$ be an arbitrarily chosen element of $A$ that is not covered by $M_{i-1}$; there is at least one such element as $M_{i-1}$ is a matching with $i-1$ edges and $2(i-1) < |A|$. Second, let $v_i$ be a uniformly random element of $[n] \setminus \{u_i\}$ that is not covered by $M_{i-1}$ and let $M_i = M_{i-1} \cup \{\{u_i, v_i\}\}$, so that $M_i$ is a matching comprising the $i$ edges $\{u_1, v_1\}, \dotsc, \{u_i, v_i\}$. Finally, let $M = M_{\lceil |A|/2 \rceil} \cup M'$, where $M'$ is a uniformly random perfect matching of the vertices of $K_n$ that are left uncovered by $M_{\lceil |A|/2 \rceil}$. Observe that for each $i \in \{1, \dotsc, \lceil|A|/2\rceil\}$,
  \[
    \Pr\big(v_i \in B \mid u_1, v_1, \dotsc, u_{i-1}, v_{i-1}, u_i\big) = \frac{|B \setminus \{v_1, \dotsc, v_{i-1}\}|}{n-2i+1} \ge \frac{|B|}{2n}.
  \]
  Thus, the number of indices $i$ for which  $v_i \in B$ can be bounded from below by a binomial random variable with parameters $\lceil |A|/2 \rceil$ and $|B|/(2n)$. Consequently, standard tail estimates for binomial distributions yield
  \[
    \Pr\left(\left|\big\{i : v_i \in B\big\}\right| \le \frac{|A||B|}{8n} \right) \le \exp\left(-\frac{|A||B|}{32n}\right),
  \]
  which implies the assertion of the proposition, as $u_i \in A$, $v_i \in B$, and $\{u_i, v_i\} \in M$ for every $i$.
\end{proof}

\begin{lemma}
  \label{lemma:cycles-perm}
  Let $M$ be a uniformly random perfect matching in $K_n$. Then for every fixed perfect matching $M'$ in $K_n$,
  \[
    \Pr\big(\text{$M \cup M'$ has more than $2\sqrt{n}$ connected components}\big) \le 2^{-\sqrt{n}/2}.
  \]
\end{lemma}
\begin{proof}
  Observe first that $M \cup M'$ is a union of even cycles and isolated edges (the edges in $M \cap M'$). We will view the isolated edges as cycles of length two so that the number of connected components of $M \cup M'$ equals the number of its cycles. Note that $M$ can be represented as $\{u_1, v_1\}, \dotsc, \{u_{n/2}, v_{n/2}\}$, where for each $i \in [n/2]$, the ordered pair $(u_i, v_i)$ is a uniformly random pair of distinct vertices of $K_n \setminus \{u_1, v_1, \dotsc, u_{i-1}, v_{i-1}\}$. The crucial observation is that after we condition on $u_1, v_1, \dotsc, u_{i-1}, v_{i-1}$ and $u_i$, there is exactly one (out of $n-2i+1$) choice for $v_i$ such that $\{u_i, v_i\}$ closes a cycle in the graph $M' \cup \{u_1, v_1\} \cup \dotsb \cup \{u_{i-1}, v_{i-1}\}$; this unique $v_i$ is the endpoint of the longest path (in the above graph) that starts at $u_i$. Consequently, the number $X$ of cycles in $M \cup M'$ has the same distribution as the sum of $n/2$ independent Bernoulli random variables $X_1, \dotsc, X_{n/2}$, where $\E[X_j] = 1/(2j-1)$. In particular,
  \begin{align*}
    \Pr\left(X \ge 2\sqrt{n}\right) & \le \Pr\left(\sum_{j=\sqrt{n}+1}^{n/2} X_j \ge \sqrt{n} \right) \le \binom{n/2-\sqrt{n}}{\sqrt{n}} \cdot \left(\frac{1}{2\sqrt{n}+1}\right)^{\sqrt{n}} \\
    & \le \left(\frac{en/2}{\sqrt{n}} \cdot \frac{1}{2\sqrt{n}}\right)^{\sqrt{n}} = \left(\frac{e}{4}\right)^{\sqrt{n}} \le 2^{-\sqrt{n}/2}.\qedhere
  \end{align*}
\end{proof}

\subsection{Eliminating potential null vectors that are almost constant} 

Recall from \cref{eq:pairanticoncentration} that we wish to use the bound 
\[
  \sup_{x\in \F_p}\Pr(W_i \cdot \boldsymbol{v} = x \mid \sigma_i) \leq \rho_{\F_p}(\boldsymbol{v}_{\sigma_i}).
\]
Note that if a vector $\boldsymbol{v}$ has large $L(\boldsymbol{v})$, then the vector $\boldsymbol{v}_{\sigma_i}$ has very small support. In this subsection, analogously to the step in the signed r.r.d.\ case where we eliminated potential null vectors with small support, we will eliminate potential null vectors with large $L(\boldsymbol{v})$. The goal of this subsection is to prove the following proposition. 

\begin{proposition}
  \label{prop:eliminate-large-levelset}
  If $Q_n$ is a uniformly random element of $\QQ_n$, then
  \[
    \Pr\left(\exists \boldsymbol{v} \in \F_p^n \setminus\{\boldsymbol{0}\} \text{ such that } Q_n\boldsymbol{v}=\boldsymbol{0} \text{ and } L(\boldsymbol{v}) \geq n-n^{0.8}\right)  \lesssim 2^{-n/100}.
  \]
\end{proposition}
\begin{proof}
  Let $\boldsymbol{L}$ denote the set of all $\boldsymbol{v} \in \F_p^n \setminus \{\boldsymbol{0}\}$ with $L(\boldsymbol{v}) \ge n-n^{0.8}$ and note that
  \[
    |\boldsymbol{L}| \le \binom{n}{n^{0.8}} \cdot p^{n^{0.8}+1} \le n^{n^{0.8}} p^{n^{0.8}+1} \lesssim 2^{2n^{0.9}},
  \]
  as $n \ll p$ and $p \le 2^{n^{0.1}}$. Therefore, the assertion of the proposition will follow from a simple union bound if we show that
  \[
    \sup_{\boldsymbol{v} \in \boldsymbol{L}} \Pr(Q_n \boldsymbol{v} = \boldsymbol{0}) \le 0.99^n.
  \]
  
  Fix an arbitrary $\boldsymbol{v} \in \boldsymbol{L}$. If $L(\boldsymbol{v}) = n$, then the supremum above is zero as the assumption that $p > n/2$ implies that $Q_n \boldsymbol{v}$ is a nonzero multiple of the all-ones vector, so we may assume that $n - n^{0.8} \le L(\boldsymbol{v}) < n$. Consider the representation of $Q_n$ as $(\boldsymbol{\sigma}, \boldsymbol{\xi})$ described in the previous subsection and fix an $i \in [n]$. Recall from~\cref{eq:pairanticoncentration} that $\Pr(W_i \cdot \boldsymbol{v} = 0 \mid \sigma_i) \le \rho_{\F_p}(\boldsymbol{v}_{\sigma_i})$. Since $\rho_{\F_p}(\boldsymbol{v}_{\sigma_i}) \le 1/2$ as long as $\boldsymbol{v}_{\sigma_i} \neq \boldsymbol{0}$, we have
  \[
    \Pr(W_i \cdot \boldsymbol{v} = 0) \le 1 - \Pr(\boldsymbol{v}_{\sigma_i} \neq \boldsymbol{0})/2.
  \]
  Let $x \in \F_p$ be the unique element for which the set $A = \{i \in [n] : v_i = x\}$ has $L(\boldsymbol{v})$ elements and let $B = [n] \setminus A$. Since $|\supp(\boldsymbol{v}_{\sigma_i})|$ is at least as large as the number of edges of the matching induced by $\sigma_i$ that have exactly one endpoint in $A$, it follows from \cref{lemma:random-matching} that
  \[
    \Pr(\boldsymbol{v}_{\sigma_i} = \boldsymbol{0}) \le \exp\left(-\frac{L(\boldsymbol{v})(n-L(\boldsymbol{v}))}{32n}\right) \le e^{-1/33} < 0.98,
  \]
  which implies that
  \[
    \Pr(Q_n \boldsymbol{v} = \boldsymbol{0}) = \prod_{i=1}^n \Pr(W_i \cdot \boldsymbol{v} = 0) \le 0.99^n,
  \]
  as claimed.
\end{proof}

\subsection{Expanding base permutations}
In this subsection, we define the subset $\EE_n \subseteq (\Sigma_n)^{n}$ mentioned in the previous subsection, and show that a uniformly random $\boldsymbol{\sigma}$ belongs to this subset with very high probability. We say that $\boldsymbol{\sigma}:=(\sigma_1,\dots,\sigma_n) \in (\Sigma_n)^{n}$ belongs to $\EE_n$ if it satisfies the following two properties:
\begin{enumerate}[{label=(Q\arabic*)}]
\item
  \label{item:Q1}
  The union of any two perfect matchings of the form $\sigma_i$ and $\sigma_j$ ($i\neq j$) has at most $n^{0.6}$ connected components.  
\item
  \label{item:Q2}
  For any two disjoint subsets $A,B\subseteq [n]$ such that $n^{0.8} \leq |A|,|B|\leq n/2$, there are at most $\sqrt{n}/2$ indices $i\in [n]$ such that the perfect matching induced by $\sigma_i$ has fewer than $|A||B|/(8n)$ edges between $A$ and $B$.   
\end{enumerate}

\begin{proposition}
\label{prop:expanding-base-perm-whp}
Let $\boldsymbol{\sigma}$ be a uniformly random element of $(\Sigma_n)^{n}$. Then, 
\[
  \Pr(\boldsymbol{\sigma} \not\in \EE_n) \lesssim 2^{-\sqrt{n}/3}.
\]
\end{proposition}
\begin{proof}
  Since the coordinates of $\boldsymbol{\sigma}$ are independent, it follows from \cref{lemma:cycles-perm} and the union bound that~\ref{item:Q1} fails with probability at most $\binom{n}{2} e^{-\sqrt{n}/2}$. \cref{lemma:random-matching} implies that for every pair $A$ and $B$ and every $i \in [n]$, the probability that $\sigma_i$ has fewer than $|A||B|/(8n)$ edges between $A$ and $B$ is at most $\exp(-n^{0.8}/32)$. Since $\sigma_1, \dotsc, \sigma_n$ are independent, then
  \[
    \Pr\big(\text{\ref{item:Q2} fails to hold}\big) \leq 2^{2n} \cdot \binom{n}{\sqrt{n}/2} \cdot \exp\left(-\frac{n^{0.8}}{64}\right)^{\sqrt{n}/2} \lesssim \exp\left(-\frac{n^{1.3}}{100}\right).
  \]
This completes the proof. 
\end{proof}

\subsection{Bounding $\Pr(\mathcal{B}_\rho)$ for small $\rho$}

Throughout this subsection, we will consider a fixed $\boldsymbol{\sigma}=(\sigma_1,\dots,\sigma_n)\in \EE_n$. Recall from the proof outline that
\[
  \boldsymbol{V} = \big\{\boldsymbol{v} \in \F_p^n : L(\boldsymbol{v}) < n - n^{0.8}\big\}
\]
and that $i_0$ is an index that attains the maximum in
\[
  \max_{i\in [n]}\left(\E \left[\inf_{\boldsymbol{v} \in S_i^{\perp}\cap \boldsymbol{V}}\Pr\left(W_i\cdot \boldsymbol{v} =0 \mid S_i \right) \right]\right).
\]

Fix a $\boldsymbol{v} \in \boldsymbol{V}$. Recall that for every $i \in [n]$, we defined the $n/2$-dimensional vector $\boldsymbol{v}_{\sigma_i}$ to be the vector whose coordinates are  $(v_{\sigma_i(2k-1)}-v_{\sigma_i(2k)})/2$. Let $T_{\boldsymbol{v}}$ denote the set of all coordinates $i \in  [n] \setminus \{i_0\}$ such that $|\supp(\boldsymbol{v}_{\sigma_i})| \ge n^{0.8}/16$. We claim that $|T_{\boldsymbol{v}}| \ge n - \sqrt{n}$. To see this, note first that the assumption that $L(\boldsymbol{v}) < n-n^{0.8}$ implies that there are disjoint sets $A_{\boldsymbol{v}}, B_{\boldsymbol{v}} \subseteq [n]$ such that $|A_{\boldsymbol{v}}| = n^{0.8}$, $|B_{\boldsymbol{v}}| = n/2$, and $v_i \neq v_j$ for all $i \in A_{\boldsymbol{v}}$ and $j \in B_{\boldsymbol{v}}$. Property~\ref{item:Q2} from the definition of $\EE_n$ implies that for all but at most $\sqrt{n}/2$ indices $i \in [n] \setminus \{i_0\}$, the perfect matching induced by $\sigma_i$ has at least $n^{0.8}/16$ edges with one endpoint in each of $A_{\boldsymbol{v}}$ and $B_{\boldsymbol{v}}$. It is easy to see that each such index $i$ belongs to $T_{\boldsymbol{v}}$.

Recall the definitions of the set $\boldsymbol{H}_t$ of $t$-good vectors and of the goodness function $h$ given in \cref{sec:good-bad-vectors}. The following is an adaptation of \cref{def:Btell-rrd} to the context of expanding base permutations.

\begin{definition}
  \label{def:Bt-rrm}
  For any $t>0$, define the set $\boldsymbol{B}_t$ of \emph{$t$-bad} vectors by
  \[
    \boldsymbol{B}_t:= \big\{\boldsymbol{v}\in \boldsymbol{V} : |\{i \in T_{\boldsymbol{v}} : \boldsymbol{v}_{\sigma_i} \notin \boldsymbol{H}_{t}\}| \ge 2 \big\}.
  \]
  We say that a pair $(i_1, i_2)$ of distinct elements of of $T_{\boldsymbol{v}}$ \emph{witnesses} $\boldsymbol{v} \in \boldsymbol{B}_t$ if
  \[
    h(\boldsymbol{v}_{\sigma_{i_1}}) \ge h(\boldsymbol{v}_{\sigma_{i_2}}) \ge \max_{i \in T_{\boldsymbol{v}} \setminus \{i_0, i_1, i_2\}} h(\boldsymbol{v}_{\sigma_i}).
  \]
\end{definition}

For the remainder of this subsection, let $\rho = p^{-1/2}$. Recall that our goal is to bound 
\[
\sum_{\boldsymbol{v}\in\boldsymbol{V}:\rho_{\F_{p}}(\boldsymbol{v}_{\sigma_{i_{0}}})\geq\rho}\prod_{i\in[n]\setminus \{i_{0}\}}\rho_{\F_{p}}(\boldsymbol{v}_{\sigma_{i}}).
\]
We begin by bounding the contribution to the above sum of vectors $\boldsymbol{v}$ that are $n$-bad.
\begin{lemma} 
  \label{lem:not-bad-wrt-2}
  If $t\geq n$, then
  \[
    \sum_{\boldsymbol{v}\in \boldsymbol{B}_t}\prod_{i\in[n]\setminus \{ i_0\}}\rho_{\F_{p}}(\boldsymbol{v}_{\sigma_{i}}) \lesssim n^{-n/20}.
  \]
\end{lemma}
\begin{proof}
  It is enough to show that the contribution to the above sum of $\boldsymbol{v} \in \boldsymbol{B}_t$ that are witnessed by a given pair $(i_1, i_2)$ of distinct indices in $[n]\setminus \{i_0\}$ is $O(n^{-n/20-2})$ and then take the union bound over all such pairs. Let us now fix such a pair for the remainder of the proof.  Given an integer $m$, let $\boldsymbol{C}_m$ be the set of all vectors $\boldsymbol{v} \in \boldsymbol{B}_t$ that are witnessed by our pair and for which $2^mt < h(\boldsymbol{v}_{\sigma_{i_2}}) \le 2^{m+1}t$. Since $h(\boldsymbol{a}) \le p$ for every vector $\boldsymbol{a}$ with $|\supp(\boldsymbol{a})| \ge |\boldsymbol{a}|^{1/4}$, then the set $\boldsymbol{C}_m$ is empty unless $t \le 2^m t \le p$ and hence $0 \le m \le \log_2 p \le n^{0.1}$.

  Fix such an $m$ and suppose that $\boldsymbol{v} \in \boldsymbol{C}_m$. It follows from the definition of a witnessing sequence that $h(\boldsymbol{v}_{\sigma_{i_1}}) \ge h(\boldsymbol{v}_{\sigma_{i_2}})$ and hence neither $\boldsymbol{v}_{\sigma_{i_1}}$ nor $\boldsymbol{v}_{\sigma_{i_2}}$ belong to $\boldsymbol{H}_{2^mt}$. It thus follows from \cref{lemma:counting-bad} that both the vectors $\boldsymbol{v}_{\sigma_{i_1}}$ and $\boldsymbol{v}_{\sigma_{i_2}}$ belong to a set of size at most $(4p)^{n/2} (2^mt)^{-n/2+n^{1/4}}$.  We next bound the number of vectors $\boldsymbol{v} \in \boldsymbol{V}$ with a given value of $(\boldsymbol{v}_{\sigma_{i_{1}}},\boldsymbol{v}_{\sigma_{i_{2}}})$. Note that all such vectors $\boldsymbol{v}$ have the same differences between all those pairs of coordinates that are connected by an edge in the union of the matchings induced by $\sigma_{i_1}$ and $\sigma_{i_2}$. In particular, the vector $\boldsymbol{v}$ is uniquely determined once we fix the value of a single coordinate in each conencted component of this graph. Since property~\ref{item:Q1} from the definition of $\EE_n$ implies that the number of connected components does not exceed $n^{0.6}$, we may conclude that
  \[
    |\boldsymbol{C}_m| \le p^{n^{0.6}} \cdot \left((4p)^{n/2} (2^mt)^{-n/2+n^{1/4}}\right)^2 \le \left(\frac{4p}{2^mt}\right)^{n} \cdot (2^mtp)^{n^{0.6}} \lesssim \left(\frac{4p}{2^mt}\right)^{n} \cdot p^{2n^{0.6}}.
  \]
 
  On the other hand, it follows from the definition of a witnessing sequence that, for each $\boldsymbol{v} \in \boldsymbol{C}_m$ and every $i \in T_{\boldsymbol{v}} \setminus \{i_0, i_1, i_2\}$, we have $h(\boldsymbol{v}_{\sigma_i}) \le h(\boldsymbol{v}_{\sigma_{i_2}})$ and hence $\boldsymbol{v}_{\sigma_i} \in \boldsymbol{H}_{2^{m+1}t}$. Consequently, \cref{lem:smallballdreg} implies that $\rho_{\F_p}(\boldsymbol{v}_{\sigma_{i}}) \leq 2^{m+1}Ct/(pn^{1/16})$ for some absolute constant $C$. In particular, every $\boldsymbol{v} \in \boldsymbol{C}_m$ satisfies
  \[
    \prod_{i \in [n] \setminus \{i_0\}} \rho_{\F_p}(\boldsymbol{v}_{\sigma_i}) \le \prod_{i \in T_{\boldsymbol{v}} \setminus \{i_0, i_1, i_2\}} \rho_{\F_p}(\boldsymbol{v}_{\sigma_i}) \le \left(\frac{2^{m+1}Ct}{pn^{1/16}}\right)^{|T_{\boldsymbol{v}}| - 2} \le \left(\frac{2^{m+1}Ct}{pn^{1/16}}\right)^{n-3\sqrt{n}}.
  \]

  Putting everything together, we see that
  \begin{align*}
    \sum_{\boldsymbol{v} \in\boldsymbol{C}_m} \prod_{i \in [n]\setminus \{i_0\}}\rho_{\F_p}(\boldsymbol{v}_{\sigma_i}) & \lesssim \left(\frac{4p}{2^mt}\right)^{n} \cdot p^{2n^{0.6}} \cdot \left(\frac{2^{m+1}Ct}{pn^{1/16}}\right)^{n-3\sqrt{n}} \\
                                                                                                                 & \lesssim (8C)^n \cdot p^{2n^{0.6}}\cdot n^{-n/17} \\
                                                                                                                 & \lesssim (8C)^n \cdot 2^{2n^{0.7}} \cdot n^{-n/17} \lesssim n^{-n/18},
  \end{align*}
  where the penultimate inequality holds because $p \le 2^{n^{0.1}}$. Since there are at most $n^{0.1} + 1$ relevant values of $m$ and $n^2$ pairs $i_1, i_2$, the claimed upper bound follows.
\end{proof}

The next lemma bounds the number of vectors $\boldsymbol{v}$ for which $\rho(\boldsymbol{v}_{\sigma_{i_0}})$ has large atom probability.

\begin{lemma}
  \label{lemma:few-bad-wrt-1}
  The number of vectors $\boldsymbol{v}\in \F_p^{n}$ for which $\rho_{\F_p}(\boldsymbol{v}_{\sigma_{i_0}}) \ge \rho$ is $O(p^{0.77n})$.
\end{lemma}
\begin{proof}
  We partition the set of relevant vectors $\boldsymbol{v}$ into two parts depending on the size of the support of $\boldsymbol{v}_{\sigma_{i_0}}$. More precisely,
  \begin{align*}
    \boldsymbol{V_{\text{small}}} & := \{\boldsymbol{v} \in \F_p^n : |\supp(\boldsymbol{v}_{\sigma_{i_0}})| \le n^{0.8}\}, \\
    \boldsymbol{V_{\text{large}}} & := \{\boldsymbol{v} \in \F_p^n : |\supp(\boldsymbol{v}_{\sigma_{i_0}})| > n^{0.8} \text{ and } \rho_{\F_p}(\boldsymbol{v}_{\sigma_{i_0}}) \ge \rho\}.
  \end{align*}
  Note first that the number of vectors $\boldsymbol{a} \in \F_p^{n/2}$ with $|\supp(\boldsymbol{a})| \le n^{0.8}$ is at most
  \[
    \binom{n/2}{n^{0.8}} \cdot p^{n^{0.8}} \le n^{n^{0.8}} \cdot p^{n^{0.8}} \lesssim p^{2n^{0.8}}
  \]
  Since $\sigma_{i_{0}}$ is fixed, then for every $\boldsymbol{a} \in \F_p^{n/2}$, there are exactly $p^{n/2}$ vectors $\boldsymbol{v} \in \F_p^n$ for which $\boldsymbol{v}_{\sigma_{i_0}} = \boldsymbol{a}$. It follows that
  \[
    |\boldsymbol{V_{\text{small}}}| \lesssim p^{n/2} \cdot p^{2n^{0.8}} \lesssim p^{2n/3}.
  \]
  Observe that if $\boldsymbol{v} \in \boldsymbol{V_{\text{large}}}$, then $\boldsymbol{v}_{\sigma_{i_0}} \not\in \boldsymbol{H}_{\rho p}$, as otherwise \cref{lem:smallballdreg} would imply that $\rho_{\F_p}(\boldsymbol{v}_{\sigma_{i_0}}) \leq C\rho/n^{1/16}$ for some absolute constant $C$, contradicting the assumption that $\boldsymbol{v} \in \boldsymbol{V_{\text{large}}}$. In particular, \cref{lemma:counting-bad} implies that the vector $\boldsymbol{v}_{\sigma_{i_0}}$ belongs to a set of size $O((4p)^{n/2} \cdot (\rho p)^{-n/2+n^{1/4}})$. Recalling that $\rho = p^{-1/2}$, we obtain
  \[
    |\boldsymbol{V_{\text{large}}}| \lesssim p^{n/2} \cdot (4p)^{n/2} \cdot (\rho p)^{-n/2+n^{1/4}} = 2^n \cdot p^n \cdot p^{-n/4+n^{1/4}/2} \lesssim p^{0.76n},
  \]
  where the last inequality holds as $p \gg 1$. We obtain the desired conclusion by summing the obtained upper bounds on $|\boldsymbol{V_{\text{small}}}|$ and $|\boldsymbol{V_{\text{large}}}|$.
\end{proof}

We now combine \cref{lem:smallballdreg,lem:not-bad-wrt-2,lemma:few-bad-wrt-1} to derive the main result of this subsection.

\begin{proposition}
  \label{prop:rrm-bounding-Brho}
  We have
  \[
    \Pr(\mathcal{B}_{\rho}) \lesssim n^{-n/20} + p^{-n/5}.
  \]
\end{proposition}
\begin{proof}
  Recall from~\cref{eq:B-rho-union-bound} that
  \begin{equation}
    \label{eqn:rrm-boundB_rho}
    \Pr(\mathcal{B}_{\rho}) \le \sum_{\boldsymbol{v}\in \boldsymbol{B}_n}\prod_{i\in[n]\setminus \{i_0\}}\rho_{\F_{p}}(\boldsymbol{v}_{\sigma_i})+\sum_{\substack{\boldsymbol{v} \in\boldsymbol{V} \setminus \boldsymbol{B}_n \\ \rho_{\F_p}(\boldsymbol{v}_{\sigma_{i_0}}) \geq \rho}} \prod_{i\in[n]\setminus \{i_0\}}\rho_{\F_p}(\boldsymbol{v}_{\sigma_i}).
  \end{equation}
  \cref{lem:not-bad-wrt-2} states that the first term in the right-hand side of \cref{eqn:rrm-boundB_rho} is $O(n^{-n/20})$. In order to bound the second term, note that for any $\boldsymbol{v} \in \boldsymbol{V} \setminus \boldsymbol{B}_n$, there are at least $|T_{\boldsymbol{v}}|-2 \geq n-2\sqrt{n}$ indices $i \in [n]\setminus \{i_0\}$ for which $\boldsymbol{v}_{\sigma_i} \in \boldsymbol{H}_n$. \cref{lem:smallballdreg} implies that for each such index, $\rho_{\F_p}(\boldsymbol{v}_{\sigma_i}) \leq n/p$. In particular, if $\boldsymbol{v} \in \boldsymbol{V} \setminus \boldsymbol{B}_n$, then
  \[
    \prod_{i\in[n] \setminus \{i_0\}}\rho_{\F_p}(\boldsymbol{v}_{\sigma_i}) \leq (n/p)^{n-2\sqrt{n}}.
  \]
  Since \cref{lemma:few-bad-wrt-1} implies that
  \[
    |\{\boldsymbol{v}\in\boldsymbol{V}:\rho_{\F_{p}}(\boldsymbol{v}_{\sigma_{v_i}})\geq \rho\}|\leq p^{0.77n},
  \]
  it follows that the second term in \cref{eqn:rrm-boundB_rho} is bounded from above by 
  \[
    p^{0.77n} (n/p)^{n-2\sqrt{n}} \lesssim p^{-0.24n} \cdot p^{2\sqrt{n}} \cdot n^n \lesssim p^{-n/5},
  \]
  where the last inequality follows as $p \gg n$.
\end{proof}

\subsection{Proof of \cref{thm:singularity-row-regular}}

The main result of this section is now immediate. 

\begin{proof}[Proof of \cref{thm:singularity-row-regular}]
  Recall from \cref{sec:overview-row-reg} that for every positive $\rho$,
  \begin{equation}
    \label{eqn:bound-Q_n-sing}
    \Pr(Q_n\text{ is singular}) \leq \Pr(\mathcal{S}^c) + \Pr(\boldsymbol{\tau} \not\in \EE_n) + n \cdot \big( \Pr(\mathcal{B}_\rho) + \rho\big).
  \end{equation}
  We know from \cref{prop:eliminate-large-levelset} that $\Pr(\mathcal{S}^{c}) \lesssim 2^{-n/100}$, from \cref{prop:expanding-base-perm-whp} that $\Pr(\boldsymbol{\tau} \not\in \EE_n) \lesssim 2^{-\sqrt{n}/3} $, and from \cref{prop:rrm-bounding-Brho} that $\Pr(B_\rho) \lesssim n^{-n/20} +  p^{-n/5}$. Recalling that $\rho = p^{-1/2}$ and $p \ge 2^{n^{0.1}}/2$, we see that the right-hand side of \cref{eqn:bound-Q_n-sing} can be bounded from above by  $Cn \cdot 2^{-n^{0.1}/2}$ for some absolute constant $C$. This gives the desired conclusion. 
\end{proof}

\bibliographystyle{abbrv}
\bibliography{counting_inverseLO}

\appendix
\section{Proof of Hal\'asz's inequality over $\F_p$}
\label{app:halasz}
In this appendix, we prove \cref{thm:halasz-fp}. The proof follows Hal\'asz's original proof in~\cite{halasz1977estimates}. 
\begin{proof}[Proof of \cref{thm:halasz-fp}]
  Let $e_p$ be the canonical generator of the Pontryagin dual of $\F_p$, that is, the function $e_p \colon \F_p \to \C$ defined by $e_p(x) = \exp(2\pi i x / p)$. Recall the following discrete Fourier identity in $\F_p$:
  \[
    \delta_0(x) = \frac{1}{p} \sum_{r \in \F_p}e_p(rx),
  \]
  where $\delta_0(0) = 1$ and $\delta_0(x) = 0$ if $x \neq 0$. Let $\epsilon_1,\dotsc,\epsilon_n$ be i.i.d.\ Rademacher random variables. Note that for any $q \in \F_p$, 
  \begin{align*}
    \Pr\left(\sum_{j=1}^n \epsilon_j a_j = q\right) &= \E\left[\delta_0 \left(\sum_{j=1}^n \epsilon_j a_j - q\right)\right] \\
                                                    &= \E\left[ \frac{1}{p} \sum_{r\in \F_p} e_p\left(r \left(\sum_{j=1}^n \epsilon_j a_j - q\right)\right)\right] \\
                                                    & = \E\left[\frac{1}{p} \sum_{r \in \F_p} \prod_{j=1}^n e_p( \epsilon_j r a_j) e_p(-r q)\right] \\
                                                    & = \frac{1}{p} \sum_{r \in \F_p} e_p(-rq) \prod_{j=1}^n \E\big[e_p(\epsilon_j r a_j)\big].
  \end{align*}
  Since each $\epsilon_j$ is a Rademacher random variable, we have
  \[
    \E\big[e_p(\epsilon_jra_j)\big] = \exp(2\pi i ra_j/p)/2 + \exp(-2\pi i ra_j/p)/2 = \cos(2\pi r a_j/p).
  \]
  It thus follows from the triangle inequality that
  \begin{equation}
    \label{eq:Pr-cos-UB}
    \Pr\left(\sum_{j=1}^n \epsilon_j a_j = q\right) \le \frac{1}{p} \sum_{r\in \F_p} \prod_{j=1}^n \left|\cos(2 \pi r a_j/p)\right| = \frac{1}{p} \sum_{r \in \F_p} \prod_{j=1}^n \left|\cos\left(\pi r a_j/p\right) \right|,
  \end{equation}
  where the equality holds because the map $\F_p \ni r \mapsto 2r \in \F_p$ is a bijection (as $p$ is odd) and (since $x \mapsto |\cos(\pi x)|$ has period $1$ and it is therefore well defined for $x \in \R/\Z$) because $|\cos(2\pi x/p)| = |\cos(\pi (2x)/p)|$ for every $x \in \F_p$.

  Given a real number $y$, denote by $\|y\| \in [0,1/2]$ the distance between $y$ and a nearest integer. Let us record the useful inequality 
\[
  |\cos(\pi y)| \leq \exp\big(-\| y \|^2/2\big),
\]
which is valid for every real number $y$. Using this inequality to bound from above each of the $n$ terms in the right-hand side of~\cref{eq:Pr-cos-UB}, we arrive at
\begin{equation}
  \label{eqn:halasz-prelim}
  \max_{q\in \F_p}\Pr\left(\sum_{i=1}^n \epsilon_i a_i = q\right) \leq \frac{1}{p} \sum_{r\in \F_p} \exp\left(-\frac{1}{2} \sum_{i=1}^n\|r a_i/p  \|^2\right).
\end{equation}
Now, for each nonnegative real $t$, we define the following `level' set:
\[
  T_t := \left\{r \in \F_p : \sum_{j=1}^{n} \|r a_j/p  \|^2 \leq t \right\}.
\]
Since for every real $y$, we may write $e^{-y} = \int_0^\infty \mathds{1}[y \le t] e^{-t}\,dt$, then
\begin{equation}
  \label{eqn:halasz-integral}
  \sum_{r\in \F_p} \exp\left(-\frac{1}{2} \sum_{j=1}^{n} \|r a_j/p  \|^2\right) = \sum_{r \in \F_p} \int_0^\infty \mathds{1}\left[\sum_{j=1}^n \|r a_j / p \|^2 \le 2t\right] e^{-t} \, dt = \int_0^{\infty} |T_{2t}| e^{-t} \, dt.
\end{equation}
Since for every nonzero $a \in \F_p$, the map $\F_p \ni r \mapsto ra \in \F_p$ is bijective, we have
\begin{align*}
  \sum_{r \in \F_p} \sum_{j=1}^n \|r a_j/p \|^2 & = \sum_{j \in \supp(\boldsymbol{a})} \sum_{r \in \F_p} \|r a_j/p\|^2 = |\supp(\boldsymbol{a})| \sum_{r \in \F_p} \|r/p\|^2 \\
                                                & = |\supp(\boldsymbol{a})| \cdot 2\sum_{s=1}^{(p-1)/2} (s/p)^2 = |\supp(\boldsymbol{a})| \cdot \frac{p^2-1}{12p} > \frac{|\supp(\boldsymbol{a})| \cdot p}{15},
\end{align*}
where the inequality holds because $p \ge 3$ (as $p$ is an odd prime). On the other hand, it follows from the definition of $T_t$ that for every $t \ge 0$,
\[
  \sum_{r \in \F_p} \sum_{j=1}^n \|r a_j/p \|^2 \le |T_t| \cdot t + \big(p - |T_t|\big) \cdot n.
\]
This implies that $|T_t| < p$ as long as $t \le |\supp(\boldsymbol{a})|/15$.

Recall that the Cauchy--Davenport theorem states that every pair of nonempty $A, B \subseteq \F_p$ satisfies $|A+B| \ge \min\{ p, |A|+|B|-1\}$. It follows that for every positive integer $m$ and every $t \ge 0$, the iterated sumset $mT_t$ satifies $|mT_t| \ge \min\{ p, m|T_t|-m\}$. We claim that for every $m$, the iterated sumset $mT_t$ is contained in the set $T_{m^2 t}$ and thus
\[
    |T_{m^2t}| \ge \min\big\{p, m|T_t|-m\big\}.
\]
Indeed, for $r_1,\dots, r_m \in T_t$, it follows from the triangle inequality and the Cauchy--Schwarz inequality that 
\begin{align*}
  \sum_{j=1}^n \left\| \sum_{i=1}^{m} r_i a_j/ p \right\|^2 \leq \sum_{j=1}^{n} \left(\sum_{i=1}^{m} \left\|r_i a_j/p \right\|\right)^2  \leq \sum_{j=1}^{n} m \sum_{i=1}^{m} \left \|r_i a_j/p\right \|^2  \leq m^2 t.
\end{align*}
Since $|T_{m^2t}| < p$ as long as $m^2t \le |\supp(\boldsymbol{a})|/15$, we see that if $t \le 2M \le |\supp(\boldsymbol{a})|/15$, then, letting $m = \lfloor \sqrt{2M/t} \rfloor \ge 1$, we obtain
\begin{equation}
  \label{eqn:halasz-cd}
  |T_t| \le \frac{|T_{m^2t}|}{m}+1 \le \frac{\sqrt{2t} \cdot |T_{2M}|}{\sqrt{M}} + 1.
\end{equation}

We now bound the size of $T_{2M}$. First, it follows from the elementary inequality
\[
  \cos(2\pi y) \ge 1-2\pi^2\|y\|^2 \ge 1 - 20\|y\|^2,
\]
which holds for all $y\in \mathbb{R}$, that $T_{2M} \subseteq T'$, where 
\[
  T' := \left\{r \in \F_p : \sum_{j=1}^n \cos(2 \pi r a_j/p) \geq n - 40M\right\}.
\]
Second, by Markov's inequality,
\[
  |T'| \le \frac{1}{\big(n-40M\big)^{2k}} \cdot \sum_{r \in T_M'}\left(\sum_{j=1}^n\cos(2\pi ra_j/p)\right)^{2k}.
\]
Third, by our assumption that $80Mk \le n$ and since the sequence $\big(1-1/(2k)\big)^{2k}$ is increasing,
\[
  (n-40M)^{2k} = \left(1 - \frac{40M}{n}\right)^{2k} \cdot n^{2k} \ge \left(1 - \frac{1}{2k}\right)^{2k} \cdot n^{2k} \ge \frac{n^{2k}}{\sqrt{2}}
\]
Fourth, since $T' \subseteq \F_p$ and $2\cos(2\pi r a_j / p) = e_p(ra_j) + e_p(-ra_j)$, we also have
\begin{align*}
  \sum_{r \in T'} \left( \sum_{j=1}^n \cos(2 \pi r a_j / p)\right)^{2k} &\le \sum_{r \in \F_p} \left( \sum_{j=1}^n \big(e_p(ra_j) + e_p(-ra_j)\big)/2\right)^{2k}\\
                                                                          &= \frac{1}{2^{2k}} \sum_{(\sigma_1,\dotsc, \sigma_{2k})\in\{\pm 1\}^{2k}} \sum_{j_1, \dots, j_{2 k}} \sum_{r \in F_p} e_p\left( r \sum_{\ell=1}^{2 k} \sigma_\ell a_{j_\ell}\right) \\
                                                                          &= \frac{1}{2^{2k}}\sum_{(\sigma_1,\dots, \sigma_{2k})\in\{\pm 1\}^{2k}} \sum_{j_1,\dots, j_{2k}} p \cdot \delta_0\left(\sum_{\ell=1}^{2k} \sigma_\ell a_{j_\ell}\right)\\
                                                                          &= \frac{p R_k(\boldsymbol{a})}{2^{2k}}.
\end{align*}
Thus, we may conclude that
\begin{equation}
  \label{eqn:halasz-moment}
  |T_M| \le |T'| \le \frac{\sqrt{2} p R_k(\boldsymbol{a})}{2^{2k}n^{2k}}
\end{equation}

Finally, combining this with \cref{eqn:halasz-prelim,eqn:halasz-integral,eqn:halasz-cd,eqn:halasz-moment}, we get, 
\begin{align*}
  \max_{q\in\F_{p}}\Pr\left(\sum_{j=1}^n\epsilon_ja_j = q\right) & \leq \frac{1}{p}\int_{0}^{M} |T_{2t}| e^{-t}\,dt+\frac{1}{p}\int_{M}^\infty pe^{-t} \, dt\\
 & \le \frac{1}{p}\int_0^{M} \left(\frac{\sqrt{2t} \cdot |T_M|}{\sqrt{M}}+1\right)e^{-t}\,dt+e^{-M} \\
 & \le \frac{|T_{2M}|}{p\sqrt{M}} \cdot \int_0^{M} \sqrt{t}e^{-t} \, dt+ \frac{1}{p} \int_0^{M} e^{-t}\, dt +e^{-M}\\
 & \leq \frac{|T_{2M}|}{p\sqrt{M}} \cdot C' + \frac{1}{p}+e^{-M}\\
                                                                 & \le \frac{CR_{k}(\boldsymbol{a})}{2^{2k}n^{2k}\sqrt{M}} + \frac{1}{p}+e^{-M},
\end{align*}
as desired.
\end{proof}

\end{document}